\documentclass[sn-mathphys-num]{sn-jnl}
\usepackage{amsmath, amsthm, amssymb}
\usepackage{epsfig}
\usepackage{graphicx}
\usepackage{rawfonts}
\usepackage{enumerate}
\usepackage{amssymb}
\usepackage{mathtools}
\usepackage{float} 
\usepackage{graphicx}
\usepackage{pgf,tikz,pgfplots}
\usepackage{mathrsfs}
\usetikzlibrary{arrows}
\usepackage{amsmath}
\usepackage{amsfonts}
\usepackage{amssymb}
\usepackage{amsthm}
\usepackage{graphicx}
\usepackage{booktabs}
\usepackage{caption}
\usepackage{listings}
\usepackage{setspace}
\usepackage[mathscr]{eucal}
\usepackage{pgfplots}
\usepackage{hyperref}
\usepackage{wrapfig}
\usepackage{floatflt,epsfig}
\usepackage{ dsfont }
\usepackage{amscd}
\usepackage{tikz-cd}
\usepackage{fancyhdr}
\usepackage[all]{xy}
\usepackage{latexsym}
\usepackage{amscd}
\usepackage{pifont}
\usepackage{eufrak}
\usepackage{subfig}
\usepackage{easyReview}

\def\Implies{\ifmmode\Longrightarrow \else
	\unskip${}\Longrightarrow{}$\ignorespaces\fi}
\def\implies{\ifmmode\Rightarrow \else
	\unskip${}\Rightarrow{}$\ignorespaces\fi}
\def\iff{\ifmmode\Longleftrightarrow \else
	\unskip${}\Longleftrightarrow{}$\ignorespaces\fi}

\let\:=\colon

\newcommand{\rank}{\mathop{\rm rank}\nolimits}

\newcommand{\numberset}{\mathbb}

\newcommand{\Z}{\numberset{Z}}

\def\NN{{\mathbb N}}
\def\ZZ{{\mathbb Z}}

\newcommand{\lt}{\mathop{\rm in}\nolimits}

\newcommand{\cC}{\mathcal{C}}

\newcommand{\cP}{\mathcal{P}}
\newcommand{\cH}{\mathcal{H}}

\newcommand{\cB}{\mathcal{B}}

\newcommand{\cF}{\mathcal{F}}

\newcommand{\cW}{\mathcal{W}}
\newcommand{\cR}{\mathcal{R}}
\newcommand{\cS}{\mathcal{S}}

\newcommand{\cV}{\mathcal{V}}

\newcommand{\rHP}{\mathrm{HP}}

\theoremstyle{thmstyleone}
\newtheorem{theorem}{Theorem}

\newtheorem{proposition}[theorem]{Proposition}

\newtheorem{corollary}[theorem]{Corollary}
\newtheorem{discussion}[theorem]{Discussion}
\newtheorem{example}[theorem]{Example}
\newtheorem{remark}[theorem]{Remark}

\theoremstyle{thmstylethree}
\newtheorem{definition}{Definition}
\newtheorem{algorithm}{Algorithm}

\begin{document}

\title[Shellable flag simplicial complexes of non-simple polyominoes]{Shellable flag simplicial complexes of non-simple polyominoes}

\author*[1]{\fnm{Francesco} \sur{Navarra}}\email{francesco.navarra@sabanciuniv.edu}

\affil*[1]{\orgdiv{Faculty of Engineering and Natural Sciences}, \orgname{Sabanci University}, \orgaddress{\street{Orta Mahalle, Tuzla}, \city{Istanbul}, \postcode{34956}, \country{Turkey}}}

\abstract{In this article, we explore the shellability of the flag simplicial complexes associated with non-simple and thin polyominoes. As a consequence, we establish the Cohen-Macaulayness and a combinatorial interpretation of the $h$-polynomial of the corresponding coordinate rings.}



\maketitle

\section*{Introduction}\label{sec1}

 A classic topic in commutative algebra is the study of determinantal ideals, that is, ideals generated by the $t$-minors of a generic matrix whose entries are elements in a ring. These ideals serve as a bridge between algebraic geometry, combinatorial algebra, and homological algebra and their study provides crucial insights into the structure and properties of varieties and rings; see for instance \cite{v,Bruns_Herzog,E,d,h}. In the specific case of $2$-minors of a matrix having indeterminates as entries, these ideals can be seen as special cases of the so-called \textit{polyomino ideals}. If $\cP$ is a polyomino, that is, a finite collection of cells joined edge by edge, then its polyomino ideal $I_{\mathcal{P}}$ is the binomial ideal generated by the inner 2-minors of $\mathcal{P}$. This type of ideal was introduced in 2012 by Qureshi~\cite{Qureshi}. Since then, the study of the main algebraic properties of polyomino ideals and their quotient rings $K[\mathcal{P}] = S/I_{\mathcal{P}}$, in relation to the shape of $\mathcal{P}$, has emerged as an exciting area of research. The aim is to explore the fundamental algebraic properties of $I_{\mathcal{P}}$, depending on the shape of $\mathcal{P}$. Despite the considerable interest in this field, many open problems remain unsolved so far. It is intriguing to determine which polyominoes can be identified with toric varieties (\cite[Proposition 1.1.11]{Toric V}), but a complete characterization is still unknown. In this context, it can be important to refer to the works that discuss the primality of $I_{\mathcal{P}}$, like \cite{Cisto_Navarra_closed_path,Cisto_Navarra_weakly,Cisto_Navarra_CM_closed_path,def balanced,Not simple with localization,Trento,Shikama}.
  It is worth mentioning that the polyomino ideals of simple polyominoes are prime. Roughly speaking, a simple polyomino is a polyomino without holes. The methods used to prove this are particularly interesting: in \cite{Simple equivalent balanced}, the authors show that simple and balanced polyominoes are equivalent and they use the fact that a polyomino ideal, associated to a balanced one, is prime (see \cite{def balanced}); independently of this, in \cite{Simple are prime} it is showed that polyomino ideals associated to simple polyominoes are prime, by identifying their quotient ring with toric rings of a weakly chordal graph, thus obtaining that $K[\cP]$ is a normal Cohen-Macaulay domain by \cite{Ohsug-Hibi_koszul}. Nowadays, the study is applied to multiply connected polyominoes, that are polyominoes with one or more holes. A fascinating class of non-simple polyominoes, called \textit{closed paths}, is introduced in \cite{Cisto_Navarra_closed_path}, where a characterization of their primality is given in terms of zig-zag walks (see \cite[Section 3]{Trento}). A closed path is basically a path of cells where the first one and the last one coincides and the path creates just one hole. For this class of polyominoes, some results are in \cite{Cisto_Navarra_closed_path}, \cite{Cisto_Navarra_Jahangir}, \cite{Cisto_Navarra_CM_closed_path}, \cite{Cisto_Navarra_Hilbert_series} and \cite{Cisto_Navarra_Veer}.Our interest lies in understanding the Cohen-Macaulay property and the combinatorial aspects of the $h$-polynomial of non-simple polyominoes. Recently, in \cite{Cisto_Navarra_Jahangir}, it is proved that the coordinate ring of a closed path with a zig-zag walk is Cohen-Macaulay and the $h$-polynomial is equal to the rook polynomial of $\cP$. Although that proof is rigorous, it is admittedly cumbersome and highly technical. In this paper we will re-derive these results through an easier and different approach, where the theory of simplicial complexes, including the shellable property and the McMullen-Walkup Theorem (see \cite[Corollary 5.1.14]{Bruns_Herzog}), will play a crucial role.
   
 The paper is organized as follows. In Section \ref{Section: Polyominoes and polyomino ideals} we introduced some combinatorial basics on polyominoes and the definition of the coordinate ring of a polyomino. Section \ref{Section: Grobner basis} provides some preliminary results. We firstly recall the definition of closed paths and the characterization of the primality of the related polyomino ideal. Now, let $\cP$ be a closed path polyomino. We provide a suitable monomial order $\prec_{\cP}$, which is inspired by \cite[Section 2]{Dinu_Navarra_Konig}, in order that the initial ideal $\mathrm{in}(I_{\cP})$ of $I_{\cP}$ with respect to $\prec_{\cP}$ is squarefree and generated in degree two (Proposition \ref{Prop: Grobner basis}). Thus, it makes sense to consider the flag simplicial complex $\Delta(\cP)$, having $\mathrm{in}(I_{\cP})$ as Stanley-Reisner ideal. In Section \ref{Section: Shellability} we investigate the combinatorial property of $\Delta(\cP)$,  particularly its pureness and shellability.  When $I_{\cP}$ is a toric ideal, it is well-known that $\Delta(\cP)$ is shellable from \cite[Theorem 9.6.1]{Villareal}. The case when $I_{\cP}$ is non-prime, or equivalently $\cP$ has a zig-zag walk, is more challenging to study: is $\Delta(\cP)$ shellable in this case? The answer is affirmative, as shown in Theorem \ref{Thm: shelling}. The crucial parts of this result, where the reader should pay close attention, are the Discussion \ref{Discussion} and Definition \ref{Defn: lex order}, where we provide a suitable shelling order for the facets of $\Delta(\mathcal{P})$. Since every shellable simplicial complex is Cohen-Macaulay, we conclude that the coordinate ring of a closed path with a zig-zag walk is also Cohen-Macaulay (see Corollary \ref{Coro: Cohen-Macaulay}). Finally, in Section \ref{Section: Rook polynomial - final}, 
using a well-know theorem of McMullen and Walkup (see \cite[Corollary 5.1.14]{Bruns_Herzog}), we show that the class of closed paths yields a positive resolution of \cite[Conjecture 4.5]{Trento3}. Moreover, we highlight that the methodologies developed in this paper can be readily adapted to achieve analogous results for weakly closed paths (Theorem \ref{Thm: weakly}). We conclude the paper by posing several open questions (see Remark \ref{Remark: Final}).

\backmatter

	\section{Polyominoes and polyomino ideals}\label{Section: Polyominoes and polyomino ideals}

 In this section, we present various combinatorial concepts related to polyominoes and we introduce the algebras associated with them. We begin by introducing the basic concepts of polyominoes. Consider the natural partial order on $\ZZ^2$, that is, if $(i,j),(k,l)\in \Z^2$, then we say that $(i,j)\leq(k,l)$ when $i\leq k$ and $j\leq l$. Let $a=(i,j)$ and $b=(k,l)$ in $\Z^2$ with $a\leq b$. The set $[a,b]=\{(m,n)\in \Z^2: i\leq m\leq k,\ j\leq n\leq l \}$ is said to be an \textit{interval} of $\Z^2$. Moreover, if $i< k$ and $j<l$, then $[a,b]$ is a \textit{proper} interval. In such a case, $a, b$ and $c=(i,l), d=(k,j)$ are \textit{diagonal corners} and the \textit{anti-diagonal corners} of $[a,b]$, respectively. We define also $]a,b[=[a,b]\setminus\{a,b\}$, $[a,b[=[a,b]\setminus\{b\}$ and $]a,b]=[a,b]\setminus\{a\}$. If $j=l$ (resp. $i=k$), then $a$ and $b$ are in \textit{horizontal} (resp. \textit{vertical}) \textit{position}. A proper interval $C=[a,b]$ with $b=a+(1,1)$ is called a \textit{cell} of $\ZZ^2$; moreover, the elements $a$, $b$, $c$ and $d$ are said to be the \textit{lower left}, \textit{upper right}, \textit{upper left} and \textit{lower right} \textit{corners} of $C$, respectively. The set of the vertices of $C$ is $V(C)=\{a,b,c,d\}$ and the set of the edges of $C$ is $E(C)=\{\{a,c\},\{c,b\},\{b,d\},\{a,d\}\}$. More generally, if $\cS$ is a non-empty collection of cells in $\Z^2$, then $V(\cS)=\bigcup_{C\in \cS}V(C)$ and $E(\cS)=\bigcup_{C\in \cS}E(C)$. The rank of $\cS$ is the number of the cells belonging to $\cS$ and it is denoted by $\vert \cS\vert$ (some authors use $\rank(\cS)$). If $C$ and $D$ are two distinct cells of $\cS$, then a \textit{walk} from $C$ to $D$ in $\cS$ is a sequence $\cC:C=C_1,\dots,C_m=D$ of cells of $\ZZ^2$ such that $C_i \cap C_{i+1}$ is an edge of $C_i$ and $C_{i+1}$ for $i=1,\dots,m-1$. Moreover, if $C_i \neq C_j$ for all $i\neq j$, then $\cC$ is called a \textit{path} from $C$ to $D$. Denoting by $(a_i,b_i)$ the lower left corner of $C_i$ for all $i=1,\dots,m$, we say that $\cC$ has a \textit{change of direction} at $C_k$ for some $2\leq k \leq m-1$ if $a_{k-1} \neq a_{k+1}$ and $b_{k-1} \neq b_{k+1}$. We say that two cells $C$ and $D$ of $\cS$ are \textit{connected} in $\cS$ if there exists a path of cells belonging to $\cS$ from $C$ to $D$. Now, we can give the formal definition of a polyomino. A \textit{polyomino} $\cP$ is a non-empty, finite collection of cells in $\Z^2$ where any two cells of $\cP$ are connected in $\cP$. For instance, see Figure \ref{Figure: Polyomino introduction}.

\begin{figure}[h]
	\centering
	\subfloat[Non-simple]{\includegraphics[scale=0.6]{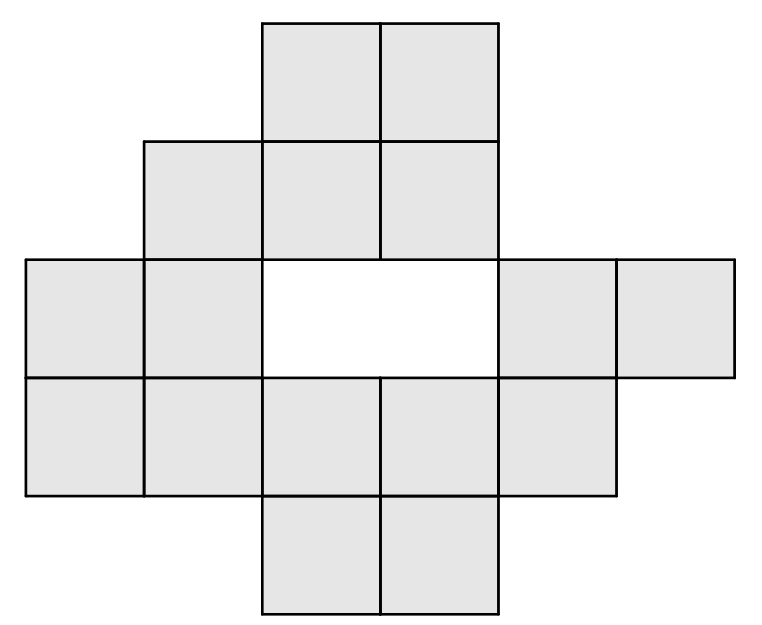}}\qquad
	\subfloat[Simple thin]{\includegraphics[scale=0.6]{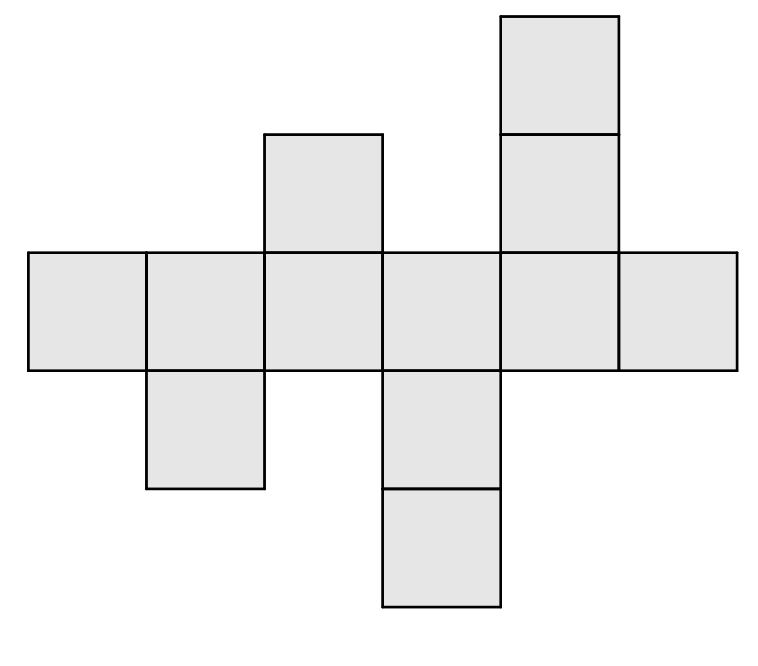}}
	\caption{Two polyominoes.}
	\label{Figure: Polyomino introduction}
\end{figure}

 Let $\cP$ be a polyomino. A \textit{sub-polyomino} of $\cP$ is a polyomino consisting of cells which belong to $\cP$. We define that $\cP$ is \textit{simple} if for any two cells $C$ and $D$ not in $\cP$ there exists a path of cells, which do not belong to $\cP$, from $C$ to $D$; look at Figure \ref{Figure: Polyomino introduction} (b) for an example of simple polyomino. A finite collection $\cH$ of cells not belonging to $\cP$ is a \textit{hole} of $\cP$ if any two cells of $\cH$ are connected in $\cH$ and $\cH$ is maximal with respect to set inclusion. For instance, the polyomino in Figure \ref{Figure: Polyomino introduction} (a) is not simple and it has only one hole which consists of two cells. Observe that every hole of $\cP$ is a simple polyomino and $\cP$ is simple if and only if $\cP$ has no hole. We say that $\cP$ is \textit{thin} if it does not contain the square tetromino, which is a square obtained as a union of four distinct cells; for example, in Figure \ref{Figure: Polyomino introduction} (b) we illustrate a simple thin polyomino. Consider two cells $A$ and $B$ of $\Z^2$ with $a=(i,j)$ and $b=(k,l)$ as the lower left corners of $A$ and $B$ with $a\leq b$. A \textit{cell interval} $[A,B]$ is the set of the cells of $\Z^2$ with lower left corner $(r,s)$ such that $i\leqslant r\leqslant k$ and $j\leqslant s\leqslant l$. If $(i,j)$ and $(k,l)$ are in horizontal (or vertical) position, we say that the cells $A$ and $B$ are in \textit{horizontal} (or \textit{vertical}) \textit{position}.\\
Let $\cP$ be a polyomino. Consider  two cells $A$ and $B$ of $\cP$ in vertical or horizontal position. 	 
A cell interval $\cB$ is called a
\textit{block of $\cP$ of rank n} if it has $n$ cells and every cell of $\cB$ belongs to $\cP$.  Moreover, a block $\cB$ of $\cP$ is \textit{maximal} if there does not exist any block of $\cP$ which properly contains $\cB$. Now, observe that if $[a,b]$ is a proper interval of $\ZZ^2$, then all the cells of $[a,b]$ identify a cell interval of $\ZZ^2$ and vice versa, that is, if $[A,B]$ is a cell interval of $\ZZ^2$ then $V([A,B])$ is a interval of $\ZZ^2$; consequently, we can associated to an interval $I$ of $\ZZ^2$ the corresponding cell interval, denoted by $\cP_{I}$. A proper interval $[a,b]$ is called an \textit{inner interval} of $\cP$ if all cells of $\cP_{[a,b]}$ belong to $\cP$. An interval $[a,b]$ with $a=(i,j)$, $b=(k,j)$ and $i<k$ is called a \textit{horizontal edge interval} of $\cP$ if the sets $\{(\ell,j),(\ell+1,j)\}$ are edges of cells of $\cP$ for all $\ell=i,\dots,k-1$. In addition, if $\{(i-1,j),(i,j)\}$ and $\{(k,j),(k+1,j)\}$ do not belong to $E(\cP)$, then $[a,b]$ is called a \textit{maximal} horizontal edge interval of $\cP$. One can similarly define a \textit{vertical edge interval} and a \textit{maximal} vertical edge interval. Following \cite{Trento}, we recall the definition of a \textit{zig-zag walk} of $\cP$. It is defined as a sequence $\cW:I_1,\dots,I_\ell$ of distinct inner intervals of $\cP$ where, for all $i=1,\dots,\ell$, the interval $I_i$ has either diagonal corners $v_i$, $z_i$ and anti-diagonal corners $u_i$, $v_{i+1}$, or anti-diagonal corners $v_i$, $z_i$ and diagonal corners $u_i$, $v_{i+1}$, such that (1) $I_1\cap I_\ell=\{v_1=v_{\ell+1}\}$ and $I_i\cap I_{i+1}=\{v_{i+1}\}$, for all $i=1,\dots,\ell-1$, (2) $v_i$ and $v_{i+1}$ are on the same edge interval of $\cP$, for all $i=1,\dots,\ell$, (3) for all $i,j\in \{1,\dots,\ell\}$ with $i\neq j$, there exists no inner interval $J$ of $\cP$ such that $z_i$, $z_j$ belong to $J$. Note that the polyomino in Figure \ref{Figure: Polyomino introduction} (a) has a zig-zag walk.\\
We conclude defining the $K$-algebra associated with a polyomino, as established by Qureshi in \cite{Qureshi}. Let $\cP$ be a polyomino and  $S_\cP=K[x_v| v\in V(\cP)]$ be the polynomial ring of $\cP$ where $K$ is a field. If $[a,b]$ is an inner interval of $\cP$, with $a$,$b$ and $c$,$d$ respectively diagonal and anti-diagonal corners, then the binomial $x_ax_b-x_cx_d$ is called an \textit{inner 2-minor} of $\cP$. The ideal $I_{\cP}$, known as the \textit{polyomino ideal} of $\cP$, is defined as the ideal in $S_\cP$ generated by all the inner 2-minors of $\cP$. We set $K[\cP] = S_\cP/I_{\cP}$, which is the \textit{coordinate ring} of $\cP$.

\section{Closed path polyominoes and reduced quadratic Gr\"obner basis} \label{Section: Grobner basis}

	In this section, we examine the reduced (quadratic) Gröbner basis of the polyomino ideal associated with a so-called closed path polyomino. In accordance with \cite{Cisto_Navarra_closed_path}, we begin by recalling its definition. We say that a polyomino $\cP$ is a \textit{closed path} if there exists a sequence of cells $A_1,\dots,A_n, A_{n+1}$, where $n>5$, such that:
	\begin{enumerate}
		\item $A_1=A_{n+1}$;
		\item $A_i\cap A_{i+1}$ is a common edge of $A_i$ and $A_{i+1}$, for all $i=1,\dots,n$;
		\item $A_i\neq A_j$, for all $i\neq j$ and $i,j\in \{1,\dots,n\}$;
		\item for all $i\in\{1,\dots,n\}$ and for all $j\notin\{i-2,i-1,i,i+1,i+2\}$, we have $V(A_i)\cap V(A_j)=\emptyset$, where $A_{-1}=A_{n-1}$, $A_0=A_n$, $A_{n+1}=A_1$ and $A_{n+2}=A_2$. 
	\end{enumerate}
	
	\begin{figure}[h]
		\centering
		\subfloat[A closed path without zig-zag walks]{\includegraphics[scale=0.5]{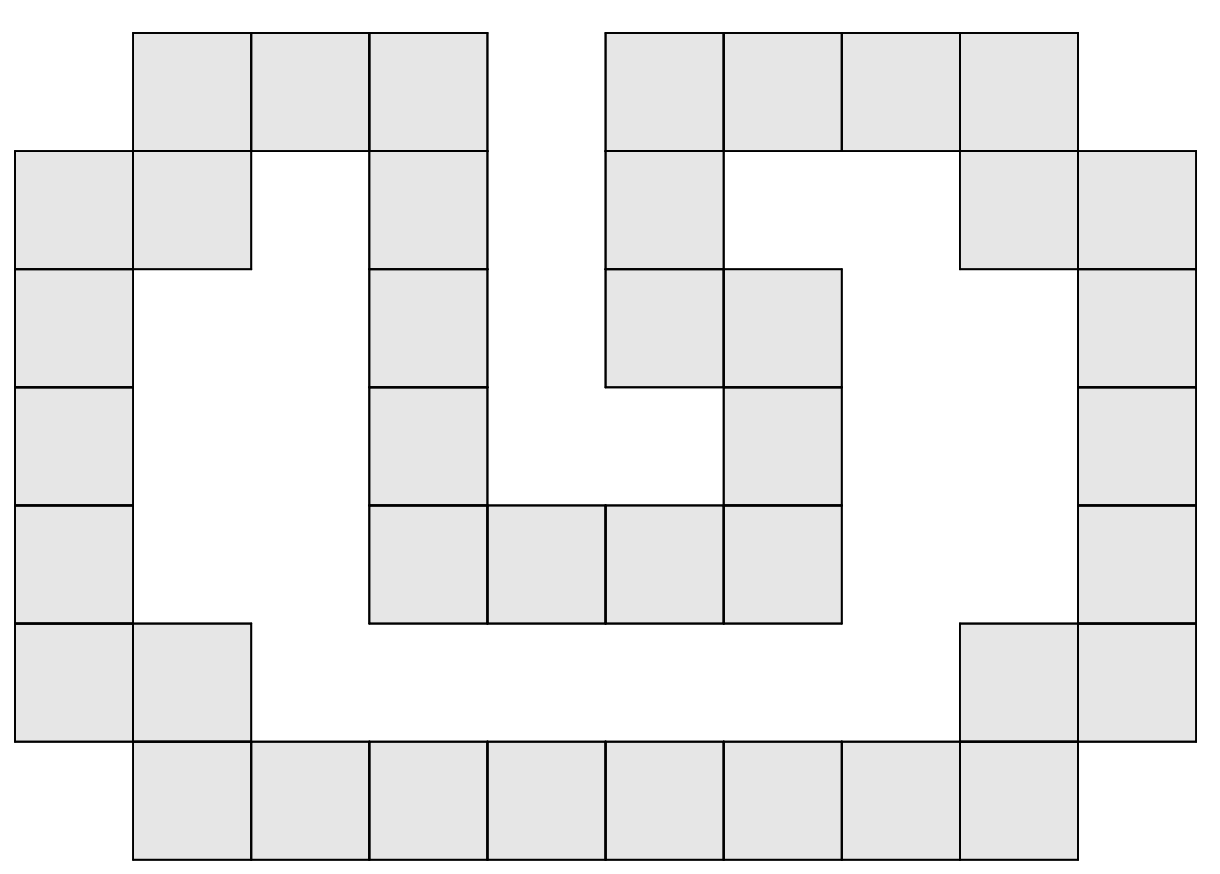}}\qquad\qquad
		\subfloat[A closed path with zig-zag walks]{\includegraphics[scale=0.5]{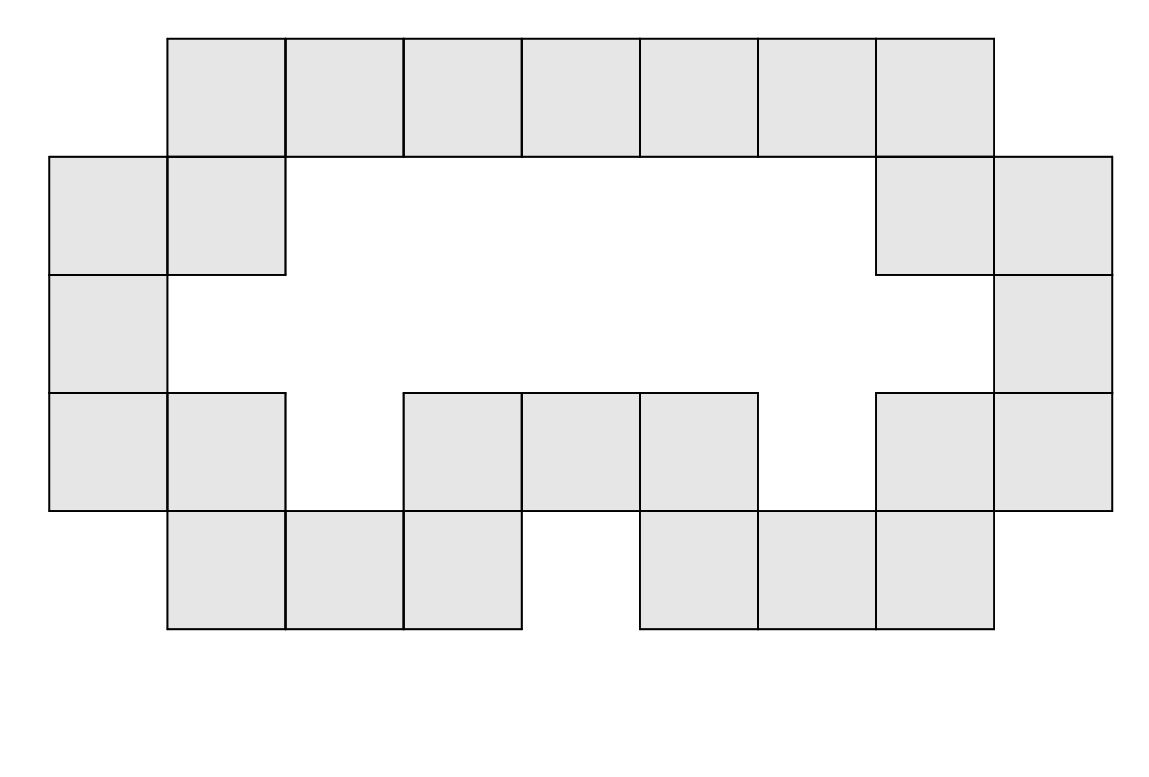}}
		\caption{An example of two closed paths.}
		\label{Figure: Example closed paths}
	\end{figure}
	
	We now present the configurations of cells that characterize their primality, specifically, an \textit{L-configuration} and a \textit{ladder of n steps}. A path of five cells $C_1, C_2, C_3, C_4, C_5$ of $\cP$ is called an \textit{L-configuration} if the two sequences $C_1, C_2, C_3$ and $C_3, C_4, C_5$ go in two orthogonal directions. A set $\cB=\{\cB_i\}_{i=1,\dots,n}$ of maximal horizontal (or vertical) blocks of rank at least two, with $V(\cB_i)\cap V(\cB_{i+1})=\{a_i,b_i\}$ and $a_i\neq b_i$ for all $i=1,\dots,n-1$, is called a \textit{ladder of $n$ steps} if $[a_i,b_i]$ is not on the same edge interval of $[a_{i+1},b_{i+1}]$ for all $i=1,\dots,n-2$. We recall that a closed path has no zig-zag walks if and only if it contains an $L$-configuration or a ladder of at least three steps (see \cite[Section 6]{Cisto_Navarra_closed_path}). For instance, in Figure \ref{Figure: Example closed paths}, the left side presents a closed path whose polyomino ideal is prime (that is, it does not contain zig-zag walks), while the right side illustrates a closed path having zig-zag walks. Finally, as easily proved in \cite[Lemma 2.1]{Dinu_Navarra_Konig}, if $\cP$ is a closed path polyomino then $|V(\cP)|=2\vert \cP\vert$.\\

	We now introduce the total order on $\{x_v:v\in V(\cP)\}$, where $\cP$ is a closed path, defined in \cite[Section 2]{Dinu_Navarra_Konig},  but with some slight generalizations.\\
	
	\begin{algorithm}\rm \label{Algorithm: to define Y}
	Let $\cP:A_1,\dots,A_n$ be a closed path polyomino. We want to label the vertices of $\cP$ in a suitable manner.
	
	\textit{First step.} Let us start to define a vertex set $Y^{(1)}\subset V(\cP)$ which will serve as the starting point of our algorithm. To accomplish this, we consider two distinct cases.
	
	\begin{enumerate}[(1)]
		\item Suppose that $\cP$ contains a configuration of four cells as in Figure \ref*{Figure: particular tetromino} (a), up to reflections or rotations of $\cP$. We consider the following two cases. 
		\begin{enumerate}
		\item Assume $A_3$ is at North of $A_2$. Then we set $Y^{(1)}=Y^{(1)}_1 \sqcup Y^{(1)}_2$ where $Y^{(1)}_1=\{1,2\}$ and $Y^{(1)}_2=\{1',2'\}$, with reference to Figure~\ref{Figure: particular tetromino} (a) if $A_4$ is at East of $A_3$ or to Figure~ \ref{Figure: particular tetromino} (b) if $A_4$ is at North of $A_3$.
		\item If $A_3$ is at East of $A_2$, then we put $Y^{(1)}=Y^{(1)}_1 \sqcup Y^{(1)}_2$ where $Y^{(1)}_1=\{1\}$ and $Y^{(1)}_2=\{1'\}$, with reference to Figure~ \ref{Figure: particular tetromino} (c). 
		\begin{figure}[h]
			\centering			
			\subfloat[]{\includegraphics[scale=0.5]{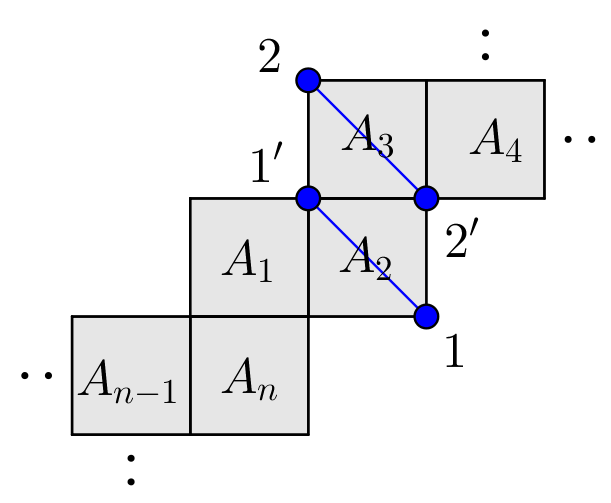}}\qquad
			\subfloat[]{\includegraphics[scale=0.5]{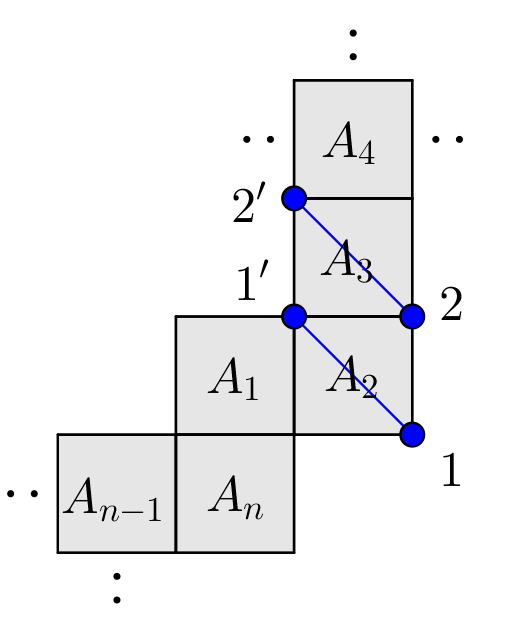}}\qquad
			\subfloat[]{\includegraphics[scale=0.5]{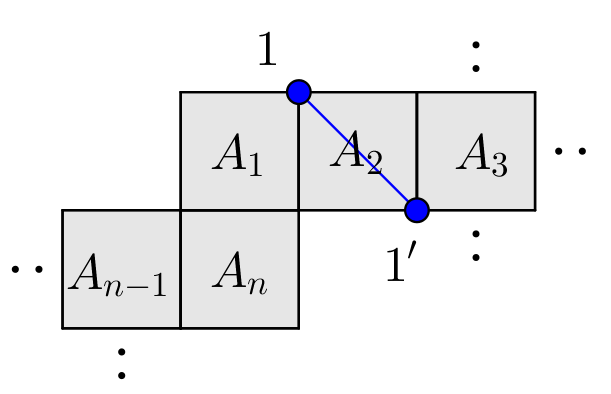}}
			\caption{\textit{First step}-(1)}
			\label{Figure: particular tetromino}
		\end{figure}
		\end{enumerate}
	
		\item Suppose that $\cP$ has an $L$-configuration in every change of direction and consider such an $L$-configuration as depicted in Figure~\ref{Figure: L configuration first step} (a), up to relabelling of the cells of $\cP$.
		\begin{enumerate}
		\item If $A_4$ is at North of $A_3$, then we set $Y^{(1)}=Y^{(1)}_1 \sqcup Y^{(1)}_2$ where $Y^{(1)}_1=\{1,2\}$ and $Y^{(1)}_2=\{1',2'\}$, referring to Figure~\ref{Figure: L configuration first step} (a). 
		\item If $A_4$ is at East of $A_3$, then we set $Y^{(1)}=Y^{(1)}_1 \sqcup Y^{(1)}_2$ where $Y^{(1)}_1=\{3\}$ and $Y^{(1)}_2=\{3'\}$, with reference to Figure~\ref{Figure: L configuration first step} (b).
		\item If $A_4$ is at South of $A_3$, then we put $Y^{(1)}=Y^{(1)}_1 \sqcup Y^{(1)}_2$ where $Y^{(1)}_1=\{1\}$ and $Y^{(1)}_2=\{1'\}$, with reference to Figure \ref{Figure: L configuration first step} (c).

			\begin{figure}[h]
			\centering
			\subfloat[]{\includegraphics[scale=0.5]{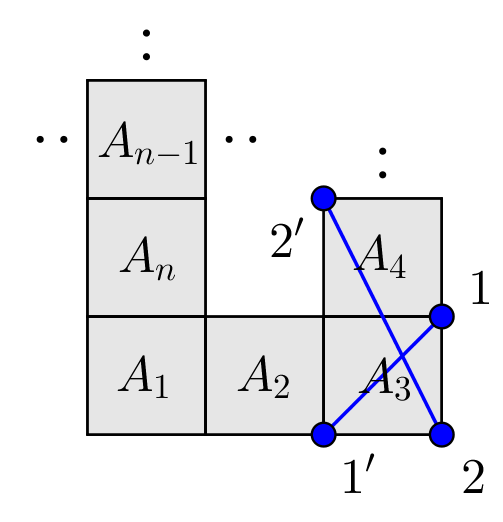}}\qquad
			\subfloat[]{\includegraphics[scale=0.5]{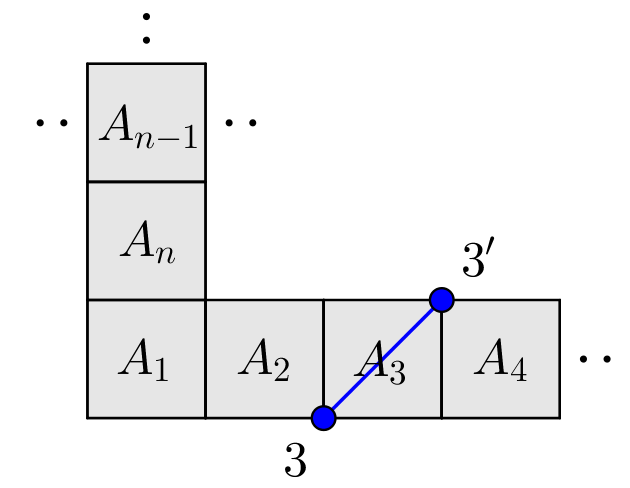}}\qquad
			\subfloat[]{\includegraphics[scale=0.5]{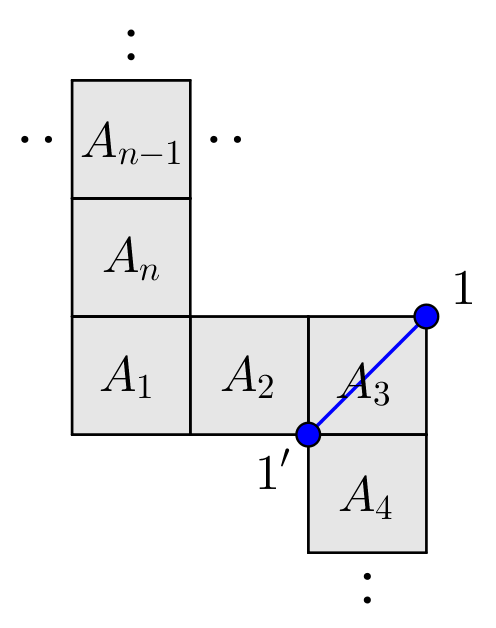}}		
			\caption{\textit{First step}-(2)}
			\label{Figure: L configuration first step}
		\end{figure}
		\end{enumerate} 
	\end{enumerate}
	
	 \textit{Halfway steps.} Now, let $j\geq 2$ and assume that $Y^{(j-1)}$ is known. We want to define $Y^{(j)}$. We refer to Table~\ref{Table2} up to suitable rotations and reflections of $\cP$. If one of the configurations in the left column of Table~\ref{Table2} occurs, where the blue vertices are in $Y^{(j-1)}$, then we denote by $k$ the maximum integer such that $m+k$ is an orange vertex in the picture displayed in the corresponding right column. Hence we define $Y^{(j)}=Y^{(j)}_1\sqcup Y^{(j)}_2$ where $Y^{(j)}_1=Y^{(j-1)}_1\sqcup \{m,\dots,m+k\}$ and $Y^{(j)}_2=Y^{(j-1)}_2\sqcup \{m',\dots,(m+k)'\}$.

\begin{table}[h]
	\centering
	\begin{minipage}{0.43\textwidth}
		\centering
		\renewcommand\arraystretch{0.9}
		\begin{tabular}{c|c|c}
			& \textbf{IF} it occurs ... & \textbf{THEN} we refer to ... \\
			\hline
			I & \begin{minipage}{0.32\textwidth}
				\includegraphics[scale=0.49]{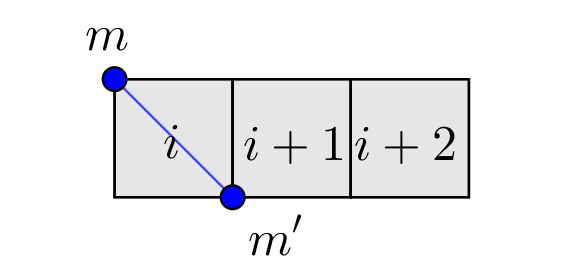}
			\end{minipage} & \begin{minipage}{0.32\textwidth}
				\includegraphics[scale=0.49]{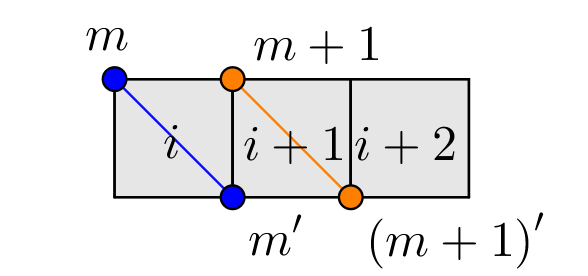}
			\end{minipage} \\
			\hline
			II & \begin{minipage}{0.32\textwidth}
				\includegraphics[scale=0.49]{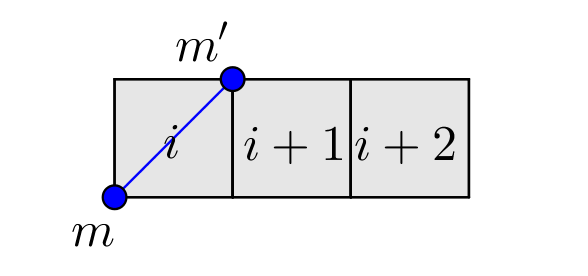}
			\end{minipage} & \begin{minipage}{0.32\textwidth}
				\includegraphics[scale=0.49]{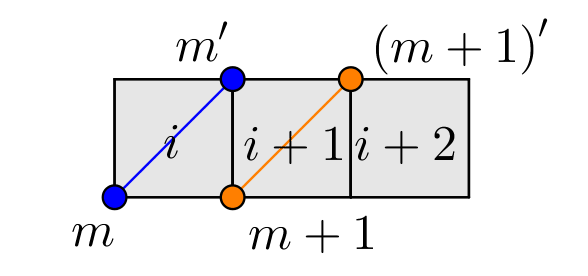}
			\end{minipage} \\
			\hline
			III & \begin{minipage}{0.32\textwidth}
				\includegraphics[scale=0.49]{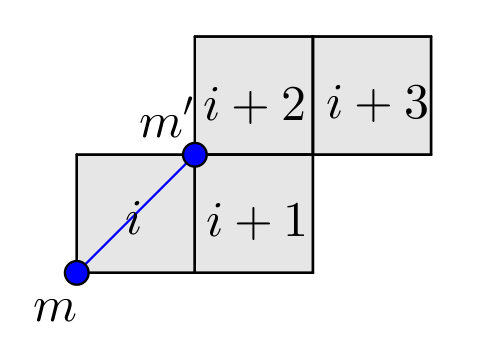}
			\end{minipage} & \begin{minipage}{0.32\textwidth}
				\includegraphics[scale=0.49]{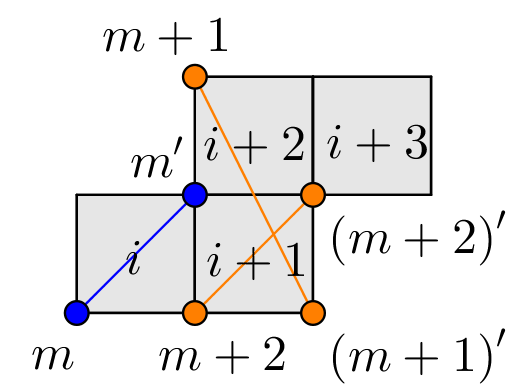}
			\end{minipage} \\
			\hline
			IV & \begin{minipage}{0.32\textwidth}
				\includegraphics[scale=0.49]{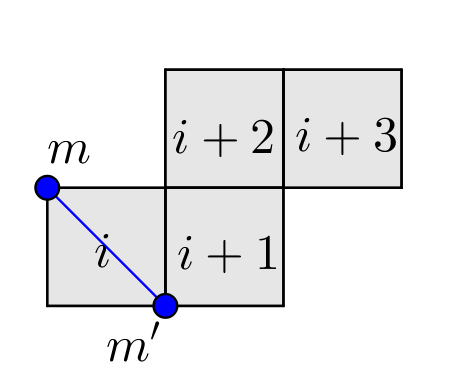}
			\end{minipage} & \begin{minipage}{0.32\textwidth}
				\includegraphics[scale=0.49]{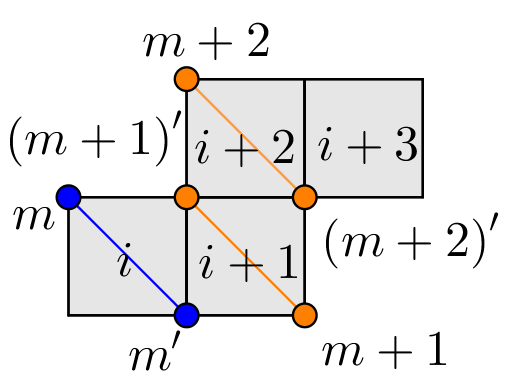}
			\end{minipage} \\
			\hline
			V & \begin{minipage}{0.32\textwidth}
				\includegraphics[scale=0.49]{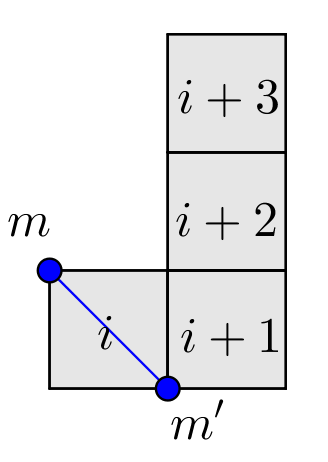}
			\end{minipage} & \begin{minipage}{0.32\textwidth}
				\includegraphics[scale=0.49]{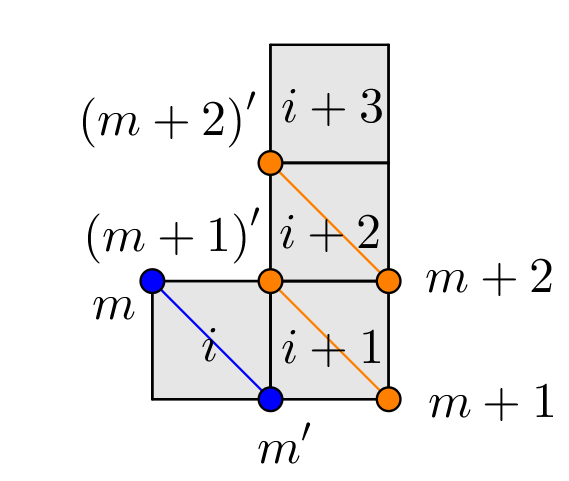}
			\end{minipage} \\
			\hline
			VI & \begin{minipage}{0.32\textwidth}
				\includegraphics[scale=0.49]{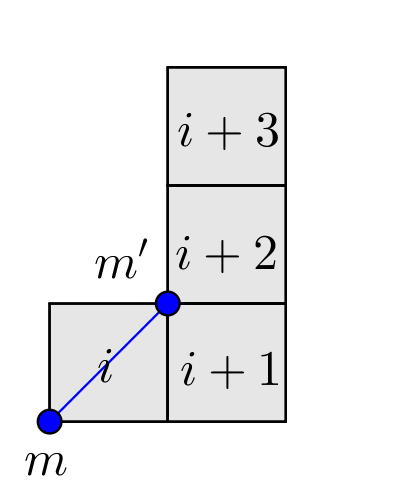}
			\end{minipage} & \begin{minipage}{0.32\textwidth}
				\includegraphics[scale=0.49]{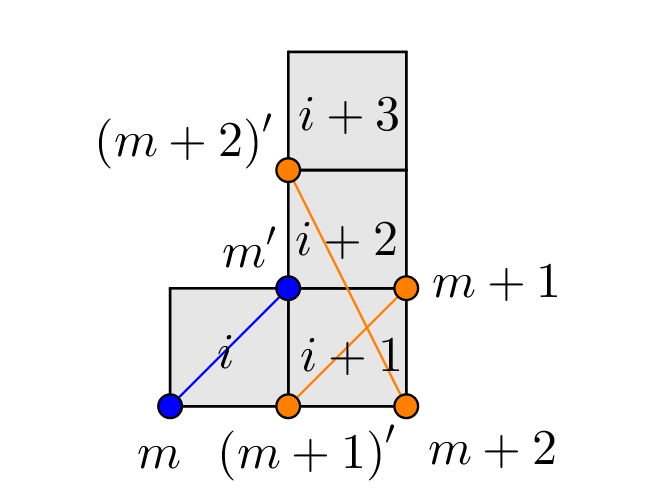}
			\end{minipage} 
		\end{tabular}
	\end{minipage}
	\hfill
	\begin{minipage}{0.43\textwidth}
		\centering
		\renewcommand\arraystretch{0.9}
		\begin{tabular}{c|c|c}
			& \textbf{IF} it occurs ... & \textbf{THEN} we refer to ... \\
			\hline
			VII & \begin{minipage}{0.34\textwidth}
				\includegraphics[scale=0.49]{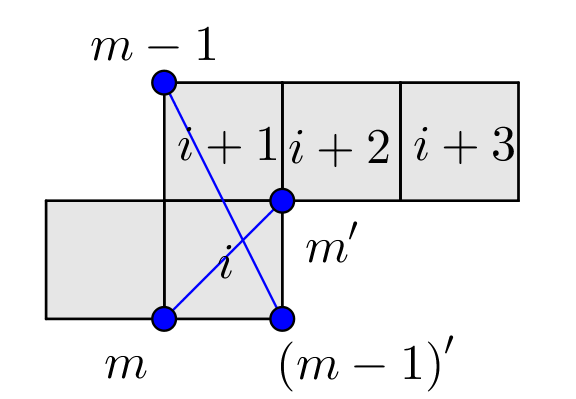}
			\end{minipage} & \begin{minipage}{0.32\textwidth}
				\includegraphics[scale=0.49]{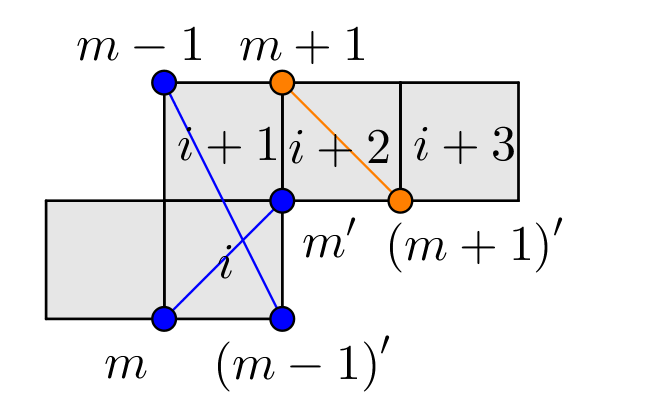}
			\end{minipage} \\
			\hline
			VIII & \begin{minipage}{0.32\textwidth}
				\includegraphics[scale=0.49]{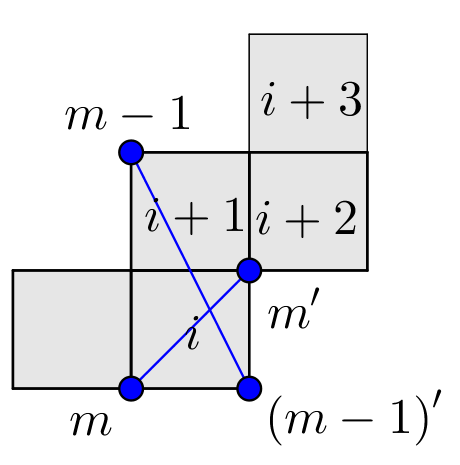}
			\end{minipage} & \begin{minipage}{0.32\textwidth}
				\includegraphics[scale=0.49]{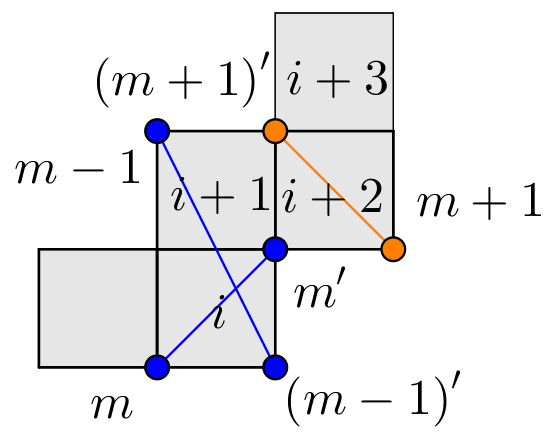}
			\end{minipage} \\
			\hline
			IX & \begin{minipage}{0.32\textwidth}
				\includegraphics[scale=0.49]{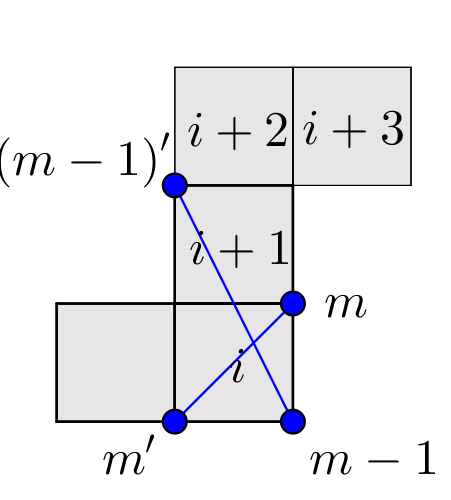}
			\end{minipage} & \begin{minipage}{0.32\textwidth}
				\includegraphics[scale=0.49]{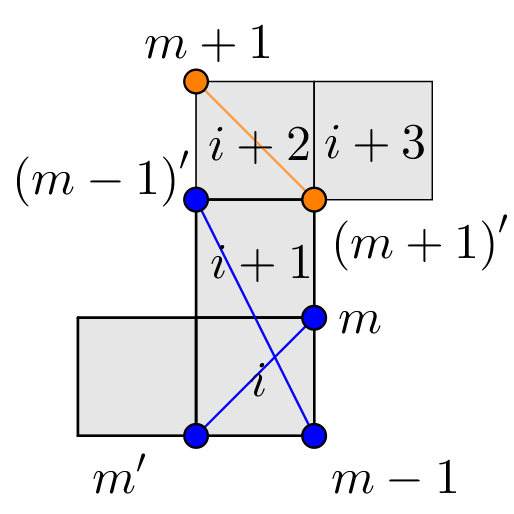}
			\end{minipage} \\
			\hline
			X & \begin{minipage}{0.20\textwidth}
				\includegraphics[scale=0.49]{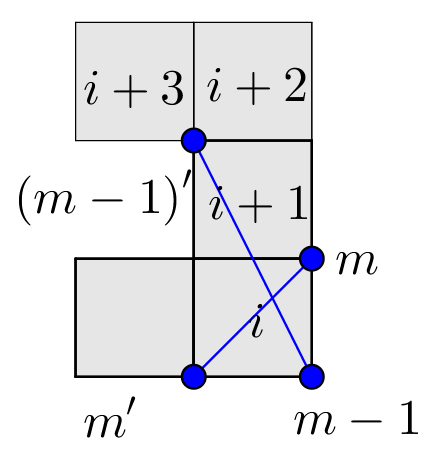}
			\end{minipage} & \begin{minipage}{0.32\textwidth}
				\includegraphics[scale=0.49]{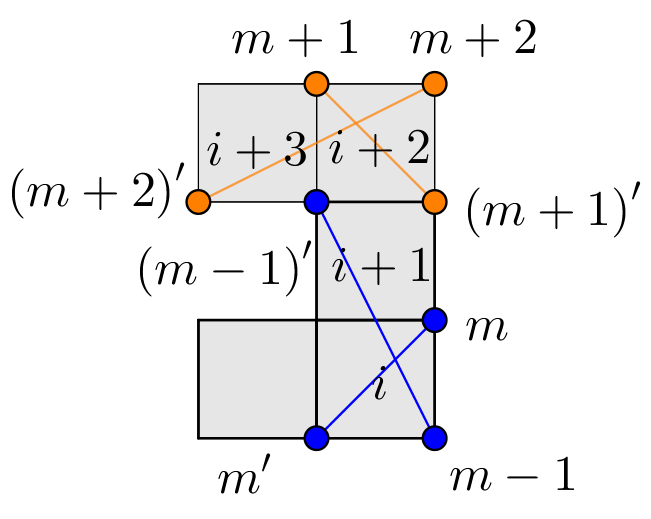}
			\end{minipage} \\
			\hline
		\end{tabular}
	\end{minipage}
	\caption{Table for labelling the vertices of a closed path.}
	\label{Table2}
\end{table}

\textit{Final step.} Let us say that $Y^{(p)}$ gives the labels of the vertices at the penultimate step of the previous procedure, just before covering all vertices of $\{A_{n-1}, A_n, A_1, A_2\}$ in pairs. We now describe the last step to obtain the set $Y$ of the labels of all the vertices of $\cP$.

	\begin{enumerate}
	\item If we start our procedure as in \textit{First step}-(1) and $A_3$ is at North of $A_2$, then we have $Y=Y_1\sqcup Y_2$ where either $Y_1=Y_1^{(p)}\sqcup \{n\}$ and $Y_2=Y_2^{(p)}\sqcup \{n\}$ with reference to Figure \ref{Figure: tetromino in the last part CASE 1} (a) or $Y_1=Y_1^{(p)}\sqcup \{n-1,n\}$ and $Y_2=Y_2^{(p)}\sqcup \{(n-1)',n'\}$ with reference to Figure \ref{Figure: tetromino in the last part CASE 1} (b). Otherwise, if $A_3$ is at East of $A_2$, then either $Y_1=Y_1^{(p)}\sqcup \{0\}$ and $Y_2=Y_2^{(p)}\sqcup \{0'\}$ referring to Figure \ref{Figure: tetromino in the last part CASE 1} (c), or $Y_1=Y_1^{(p)}\sqcup \{n,0\}$ and $Y_2=Y_2^{(p)}\sqcup \{n',0'\}$ as in Figures \ref{Figure: tetromino in the last part CASE 1} (d)-(e).
	
	\begin{figure}[h]
		\centering
		\subfloat[$a\in Y_2^{(p)}$]{\includegraphics[scale=0.5]{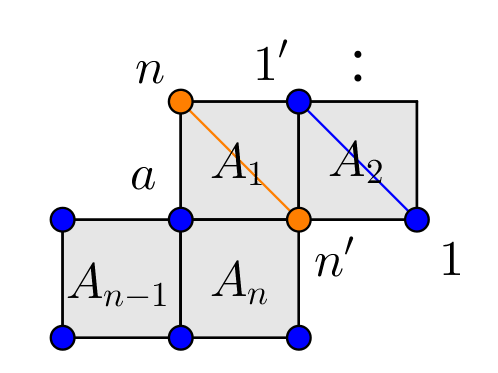}}
		\subfloat[$a\in Y_2^{(p)}$]{\includegraphics[scale=0.5]{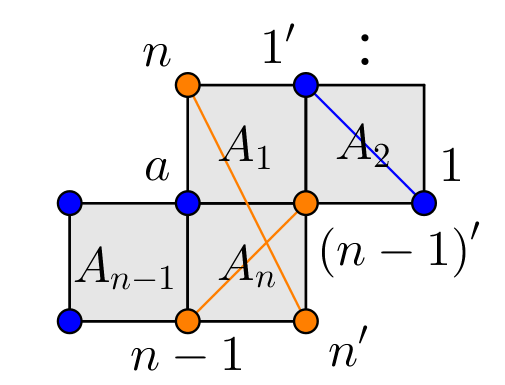}}
		\subfloat[]{\includegraphics[scale=0.5]{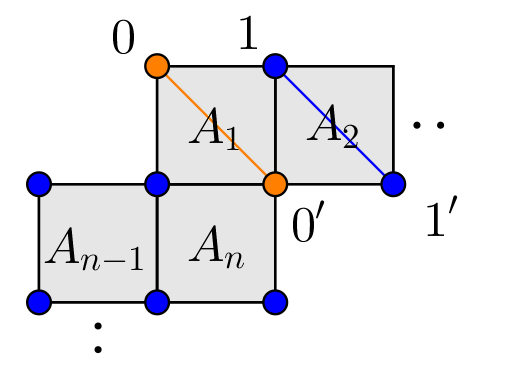}}
		\subfloat[$a\in Y_2^{(p)}$]{\includegraphics[scale=0.5]{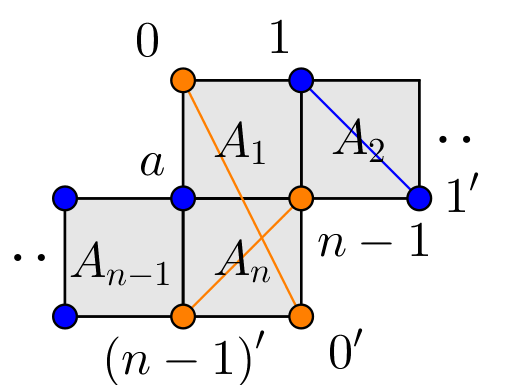}}
		\subfloat[$b\in Y_2^{(p)}$]{\includegraphics[scale=0.5]{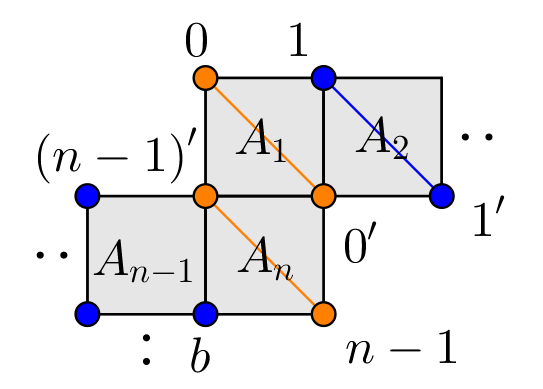}}
		\caption{Last step coming from \textit{First step}-(1)}
		\label{Figure: tetromino in the last part CASE 1}
	\end{figure}
	
	\item If our procedure starts as in \textit{First step}-(2) and $A_4$ is at North or South of $A_3$, then we have $Y=Y_1\sqcup Y_2$ where $Y_1=Y_1^{(p)}\sqcup \{n-1,n\}$ and $Y_2=Y_2^{(p)}\sqcup \{(n-1)',n'\}$ with reference to Figures \ref{Figure: L configuration A4 at north of A3} (a)-(b)-(c)-(d). In the other case, if $A_4$ is at East of $A_3$, then $Y_1=Y_1^{(p)}\sqcup \{1,2\}$ and $Y_2=Y_2^{(p)}\sqcup \{1',2'\}$ referring to Figure \ref{Figure: L configuration A4 at north of A3} (e)-(f).

\begin{figure}[h]
	\subfloat[$a \in Y_2^{(p)}$]{\includegraphics[scale=0.49]{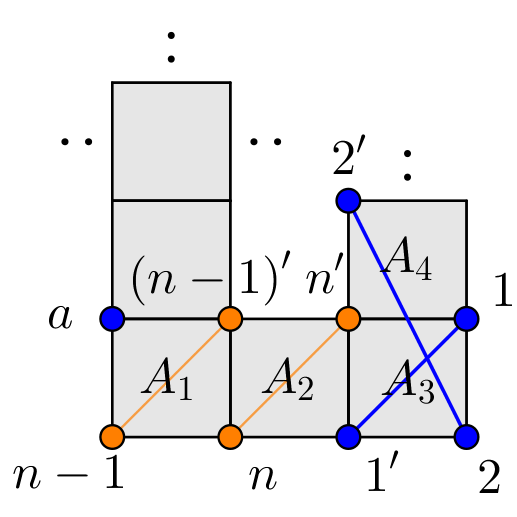}}\
	\subfloat[$b\in Y_2^{(p)}$]{\includegraphics[scale=0.49]{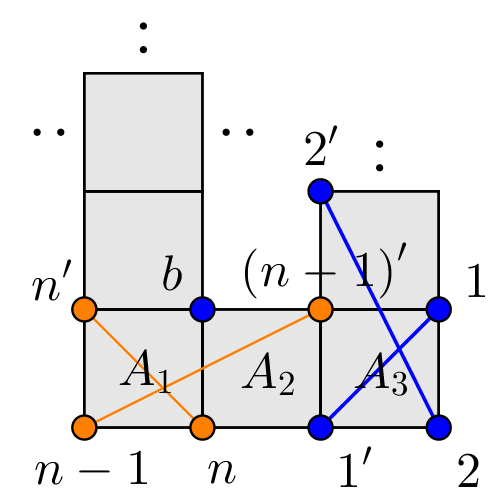}}\
	\subfloat[$a\in Y_2^{(p)}$]{\includegraphics[scale=0.49]{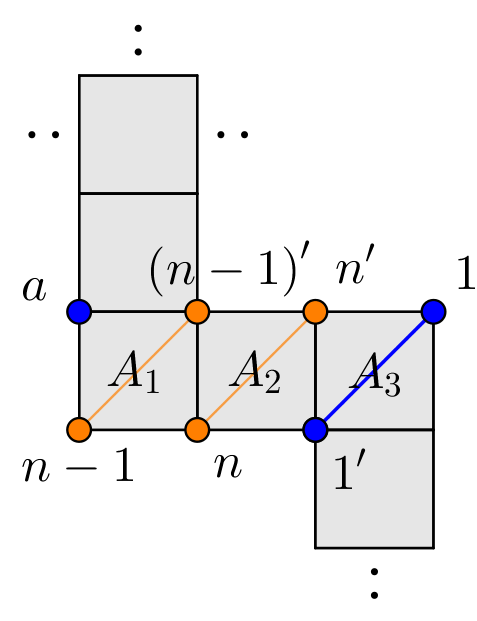}}\
	\subfloat[$b\in Y_2^{(p)}$]{\includegraphics[scale=0.49]{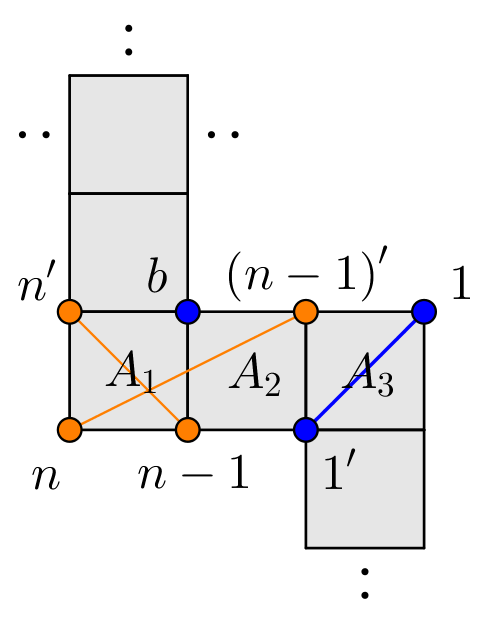}}
	\subfloat[]{\includegraphics[scale=0.49]{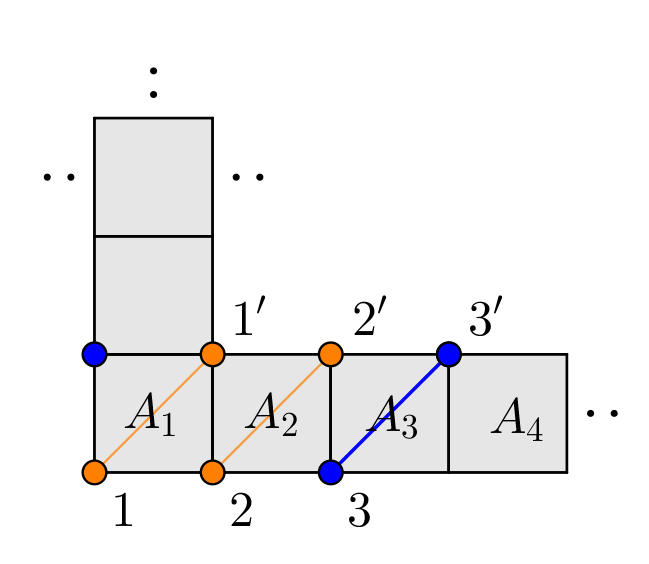}}
	\subfloat[]{\includegraphics[scale=0.49]{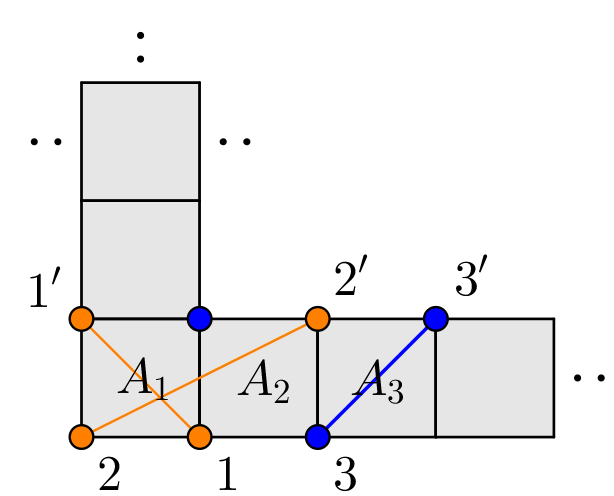}}
	\caption{Last step coming from \textit{First step}-(2)}
	\label{Figure: L configuration A4 at north of A3}
\end{figure}
\end{enumerate}

\end{algorithm}

\begin{definition}\rm\label{Defn: lex order}
	Let $\cP$ be a closed path polyomino and $Y=Y_1\sqcup Y_2$ be the set of the labels given by Algorithm \ref{Algorithm: to define Y}. We define the following total order on $\{x_v:v\in V(\cP)\}$. Denote by $<_{Y_2}$ an arbitrary total order on $Y_2$. Let $i,j\in Y$.	
	\[
	x_i <_\cP x_j  \Longleftrightarrow \left \lbrace \begin{array}{l}
		i \in Y_2 \text{ and } j\in Y_1\\
		i,j\in Y_1 \text{ and } i<j \\
		i,j \in Y_2 \text{ and } i<_{Y_2} j
	\end{array}
	\right.
	\]
	
	\noindent Set by $\prec_{\cP}$ the lexicographic order on $S_{\cP}$ induced by the total order $<_{\cP}$. \\
\end{definition}

\begin{example}\rm
	 In Figure~\ref{Figure: example closed path of konig type 2}, we illustrate two examples of labelling the vertices of a closed path, as described in Algorithm~\ref{Algorithm: to define Y}. In particular, the lexicographic orders $\prec_{\cP_1}$ and $\prec_{\cP_2}$ induced by $x_1>x_2>\dots>x_{16}>x_{1'}>x_{2'}>\dots>x_{16'}$ for $\cP_1$ and $x_1>x_2>\dots>x_{10}>x_{10'}>x_{9'}>\dots>x_{1'}$ for $\cP_2$ are examples among the many provided in Definition \ref{Defn: lex order}. Moreover, as an application of the Buchberger's criterion (see \cite[Section 2.4]{EHGrobner}), one can easily verify that the set of the generators of $I_{\cP_1}$ and $I_{\cP_2}$ forms the reduced Gr\"obner basis with respect to $\prec_{\cP_1}$ and $\prec_{\cP_2}$, respectively.  The following proposition demonstrates this for all closed paths.
	
	 \begin{figure}[h!]
		 	\centering
		 	\subfloat[Closed path $\cP_1$ ($n=16$)]{\includegraphics[scale=0.65]{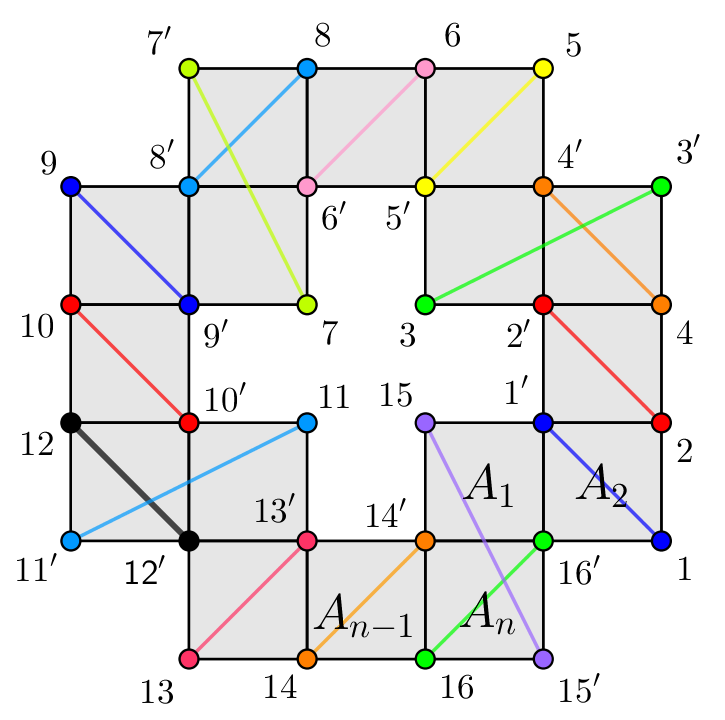}}
		 	\subfloat[Closed path $\cP_2$]{\includegraphics[scale=0.65]{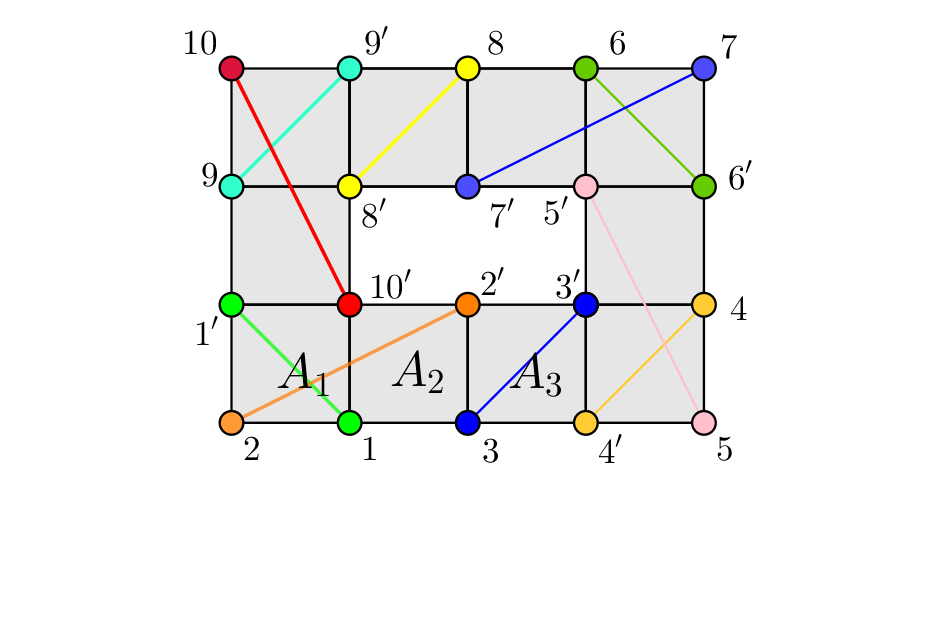}}
		 	\caption{Two closed paths and the related labels of the vertices.}
		 	\label{Figure: example closed path of konig type 2}
		 \end{figure}
\end{example}

	Just for simplicity, we denote by $\mathrm{in}(f)$  and $\mathrm{in}(I_{\cP})$ the initial monomial of a polynomial $f$ and the initial ideal of $I_{\cP}$, respectively, without mentioning the order $\prec_{\cP}$.\\  

\begin{proposition}\label{Prop: Grobner basis}
	Let $\cP$ be a closed path polyomino and $\prec_{\cP}$ be a lexicographic order given in Definition \ref{Defn: lex order}. Then the set of the generators of $I_{\cP}$ forms the reduced (quadratic) Gr\"obner basis of $I_{\cP}$ with respect to $\prec_{\cP}$.
\end{proposition}

\begin{proof}
	Let $f=x_ax_b-x_cx_d$ and $g=x_\alpha x_\beta-x_\gamma x_\delta$ be two generators of $I_{\cP}$, whose inner intervals are $[a,b]$ and $[\alpha,\beta]$, respectively. We want to prove that the $S$-polynomial $S(f,g)$ of $f$ and $g$ reduces to zero with respect to the generators of $I_{\cP}$ (see \cite[Theorem 2.14]{EHGrobner}). If $\vert\{a,b,c,d\}\cap \{\alpha,\beta,\gamma,\delta\}\vert=2$, then the claim follows from \cite[Remark 1]{Cisto_Navarra_CM_closed_path}. Assume that $\vert\{a,b,c,d\}\cap \{\alpha,\beta,\gamma,\delta\}\vert=1$. Firstly, consider that $c=\gamma$ and $\mathrm{gcd}(\mathrm{in}(f),\mathrm{in}(g))=x_c$. If $\mathrm{gcd}(\mathrm{in}(f),\mathrm{in}(g))=1$, then there is nothing further to prove, as stated in \cite[Proposition 2.15]{EHGrobner}). In such a situation, let $C=[a,b]\cap [\alpha,\beta]$ and $h$ be the vertex in $[b,d]\cap [\alpha, \delta]$, so $\alpha$ and $b$ (resp. $c$ and $h$) are the diagonal (resp. anti-diagonal) corners of $C$. After computing $S(f,g)$, which equals to $x_ax_bx_\delta-x_\alpha x_dx_\beta$, we need to analyse some cases looking at Table \ref{Table2}. 
	\begin{itemize}
		\item If we are in the case VII, then $C$ is the cell with index $i+1$, $c=m+1$, $h=m'$, $b=m+1$, $a=m$ and $d=(m-1)'$, so $\mathrm{in}(S(f,g))=x_ax_bx_\delta$ and $x_h, x_\beta<_{\cP}x_b$, that is, the condition (1) of \cite[Lemma 2 (with reference to Figure 5(b))]{Cisto_Navarra_CM_closed_path} is satisfied.
		\item When we consider the case VIII, then $C$ is the cell with index $i+1$, $c=m-1$, $h=m'$, $b=(m+1)'$, $a=m$, $d=(m-1)'$ and $\delta=m+1$, so $\mathrm{in}(S(f,g))=x_ax_bx_\delta$ and $x_h, x_\beta<_{\cP}x_\delta$, which means the condition (1) of \cite[Lemma 2 (with reference to Figure 5(b))]{Cisto_Navarra_CM_closed_path} holds.
		\item Look at the case IX. Then we have that $C$ is the cell with index $i+2$, $c=m+1$, $h=(m+1)'$, $d\in \{m-1,m\}$, and $a\in Y_2$, so $\mathrm{in}(S(f,g))=x_\alpha x_dx_\beta$ and $x_h, x_a<_{\cP}x_d$, hence the condition (3) of \cite[Lemma 2 (with reference to Figure 5(b))]{Cisto_Navarra_CM_closed_path} is verified.
	\end{itemize} 
	In every case, $S(f,g)$ reduces to zero with respect to the generators of $I_{\cP}$. By following similar arguments as before and by examining Table \ref{Table2} and all the figures in Algorithm \ref{Algorithm: to define Y}, it is straightforward to verify that  $\prec_{\cP}$ and $<_{\cP}$ satisfy the conditions (1) or (2) in \cite[Lemmas 2, 3 or 4]{Cisto_Navarra_CM_closed_path}; additionally, when $b=\alpha$, it holds that $\mathrm{gcd}(\mathrm{in}(f),\mathrm{in}(g))=1$. Therefore, in any case, $S(f,g)$ reduces to zero with respect to the generators of $I_{\cP}$, thus proving the claim.
\end{proof}

\begin{corollary}\label{Coro: radicality + CM closed path no zig-zag}
	Let $\cP$ be a closed path polyomino. Then $I_{\cP}$ is a radical ideal and $K[\cP]$ is a Koszul ring. In particular, if $\cP$ does not contains zig-zag walks, then $K[\cP]$ is a normal Cohen-Macaulay domain with Krull dimension equal to $\vert V(\cP)\vert-\vert\cP\vert$.
\end{corollary}

\begin{proof}
	The fact that $I_{\cP}$ is radical and $K[\cP]$ is Koszul follows from \cite[Lemma 6.51]{EHGrobner} and \cite[Theorem 2.28]{binomial ideals}, respectively. Moreover, if $\cP$ has no zig-zag walks, then we have that $I_{\cP}$ is a toric ideal by \cite[Theorem 6.2]{Cisto_Navarra_closed_path}. Hence, by applying \cite[Corollary 4.26]{binomial ideals} and \cite[Theorem 6.3.5]{Bruns_Herzog}, we have that  $K[\cP]$ is a normal Cohen-Macaulay ring. Moreover, from \cite[Theorem 3.5 and Corollary 3.6]{Qureshi} we know that $I_\cP$ can be viewed as the lattice ideal of a saturated lattice $\Lambda$ with $\operatorname{rank}_\mathbb{Z}(\Lambda)=|\cP|$. Therefore the height $\operatorname{ht}(I_\cP)$ of $I_{\cP}$ is $|\cP|$, so $\dim K[\cP]=|V(\cP)|-|\cP|$.
\end{proof}

Note that the monomial orders provided in Definition \ref{Defn: lex order} considerably simplify that ones defined in \cite[Section 4]{Cisto_Navarra_CM_closed_path}. Moreover, those orders are much easier to implement in the \texttt{Macaulay2} package \cite{Package_M2,M2}.\\

\section{Shellability of a simplicial complex attached to a closed path} \label{Section: Shellability}

This section is devoted to studying the shellability of the simplicial complexes associated with a closed path polyomino. We begin recalling some basic facts about simplicial complexes. A \textit{finite simplicial complex} $\Delta$ on $[n]:=\{1,\dots,n\}$ is a
collection of subsets of $[n]$ such that $\{i\}\in \Delta$ for all $i=1,\dots,n$ and, if $F'\in \Delta$ and $F \subseteq F'$, then $F \in \Delta$. The elements of $\Delta$ are called \textit{faces}. The dimension of a face $F$ is one less than its cardinality and it is denoted by $\dim(F)$. A face of $\Delta$
of dimension 0 (resp. 1) is called a \textit{vertex} (resp. \textit{edge}) of $\Delta$. The maximal faces of $\Delta$ with respect to the set inclusion are said to be the \textit{facets} of $\Delta$. The dimension of $\Delta$ is given by $\max\{\dim(F):F\in \Delta\}$. A simplicial complex $\Delta$ is \textit{pure} if all the facets have the same dimension. The dimension of a pure simplicial complex $\Delta$ is trivially given by the dimension of a facet of $\Delta$. Given a collection $\mathrm{F}=\{F_1,\dots,F_m\}$ of subsets of $[n]$, we denote by $\langle F_1,\dots,F_m\rangle$ or briefly $\langle \mathrm{F}\rangle$ the simplicial complex consisting of all the subsets of $[n]$ which are contained in $F_i$, for some $i=1,\dots,m$. This simplicial complex is said to be generated by $F_1,\dots,F_m$; in particular, if $\mathcal{F}(\Delta)$ is the set of the facets of $\Delta$, then $\Delta$ is obviously generated by $\mathcal{F}(\Delta)$. Referring to \cite[Definition 5.1.11]{Bruns_Herzog}, we say that a pure simplicial complex $\Delta$ is \textit{shellable} if its facets can be ordered as $F_1,\dots,F_m$ in such a way that $\langle F_1,\dots,F_{i-1}\rangle \cap \langle F_i\rangle$ is generated by a non-empty set of maximal proper faces of $\langle F_i\rangle$, for all $i\in \{2,\dots,m\}$, or equivalently, for all $i,j\in [n]$ with $j<i$ there exist some $v \in F_i \setminus F_j$ and some $k \in[i-1]$ such that $F_i\setminus F_k =\{v\}$. In this case $F_1,\dots,F_m$ is called a \textit{shelling} of $\Delta$. \\
 Let $\Delta$ be a simplicial complex on $[n]$ and  $R=K[x_1,\dots,x_n]$, where $K$ is a field. To every collection $F=\{i_1,\dots,i_r\}$ of $r$ distinct vertices of $\Delta$, there is an associated monomial $x_F$ in $R$ where $x_F=x_{i_1}\dots x_{i_r}.$ The monomial ideal generated by all such monomials $x_F$ where $F$ is not a face of $\Delta$ is called the \textit{Stanley-Reisner ideal} and it is denoted by $I_{\Delta}$. The \textit{face ring} of $\Delta$, denoted by $K[\Delta]$, is defined as $R/I_{\Delta}$. A simplicial complex is called \textit{flag} if all its minimal non-faces have cardinality two; in
 other words, if its Stanley-Reisner ideal is generated by square-free monomials of degree
 two. We recall that, if $\Delta$ is a simplicial complex on $[n]$ of dimension $d$, then $\dim K[\Delta]=d+1$ (\cite[Corollary 6.3.5]{Villareal}).\\
  In this context, we now introduce the flag simplicial complex attached to a closed path (with respect to $\prec_{\cP}$).\\
 
\begin{definition}\rm 
Let $\cP$ be a closed path polyomino and $\prec_{\cP}$ be a monomial order provided in Definition \ref{Defn: lex order}. Since the set of the generators of $I_{\cP}$ forms the reduced (quadratic) Gr\"obner basis of $I_{\cP}$ with respect to $\prec_{\cP}$, then $\lt(I_\cP)$ is squarefree and it is generated in degree two. We denote by $\Delta(\cP)$ the flag simplicial complex on $V(\cP)$ with $\lt(I_{\cP})$ as Stanley-Reisner ideal. We call it the \textit{simplicial complex attached to $\cP$ with respect to $\prec_{\cP}$}.	\\
\end{definition}

\begin{remark}\rm \label{Remark: structure path with zig-zag walk}
	Let $\cP$ be a closed path polyomino. If $\cP$ contains a zig-zag walk, its shape is well-known by \cite[Section 6]{Cisto_Navarra_closed_path} (see also \cite[Remark 3.3]{Dinu_Navarra_Konig} for more details). Specifically, $\cP$ can be described as a non-disjoint union of suitably rotated or reflected cell arrangements, shown in Figure \ref{Figure: example structure closed path zig-zag}. For example, if $\cP$ contains a sequence of cells similar to Figure \ref{Figure: example structure closed path zig-zag} (a), then other parts of $\cP$ can be constructed by rotating or reflecting one of the cell configurations in Figure \ref{Figure: example structure closed path zig-zag} to overlap $\{A,B,C,D,E\}$ with $\{P,Q,R,S,T\}$.\\
	Consider the closed path on the left in Figure \ref{Figure: Example closed paths}. It is formed by joining a configuration from Figure \ref{Figure: example structure closed path zig-zag} (c) (where $E=P$) with another configuration from Figure \ref{Figure: example structure closed path zig-zag} (c) (where $E\neq P$), rotated 90 degrees counter-clockwise. This is then connected to another configuration from Figure \ref{Figure: example structure closed path zig-zag} (c) (where $E=P$) rotated 180 degrees counter-clockwise. Finally, the initial and latter configurations are connected by a cell arrangement from Figure \ref{Figure: example structure closed path zig-zag} (a), rotated 90 degrees clockwise.\\
	Moreover, observe that the vertex labelling method from	$D$ to $R$ remains consistent across every cell arrangement illustrated in Figure \ref{Figure: example structure closed path zig-zag}. Look at Figure \ref{Figure: example closed path of konig type 2} (a) for a concrete example.\\ 
	
	\begin{figure}[h!]
		\centering
		\subfloat[]{\includegraphics[scale=0.5]{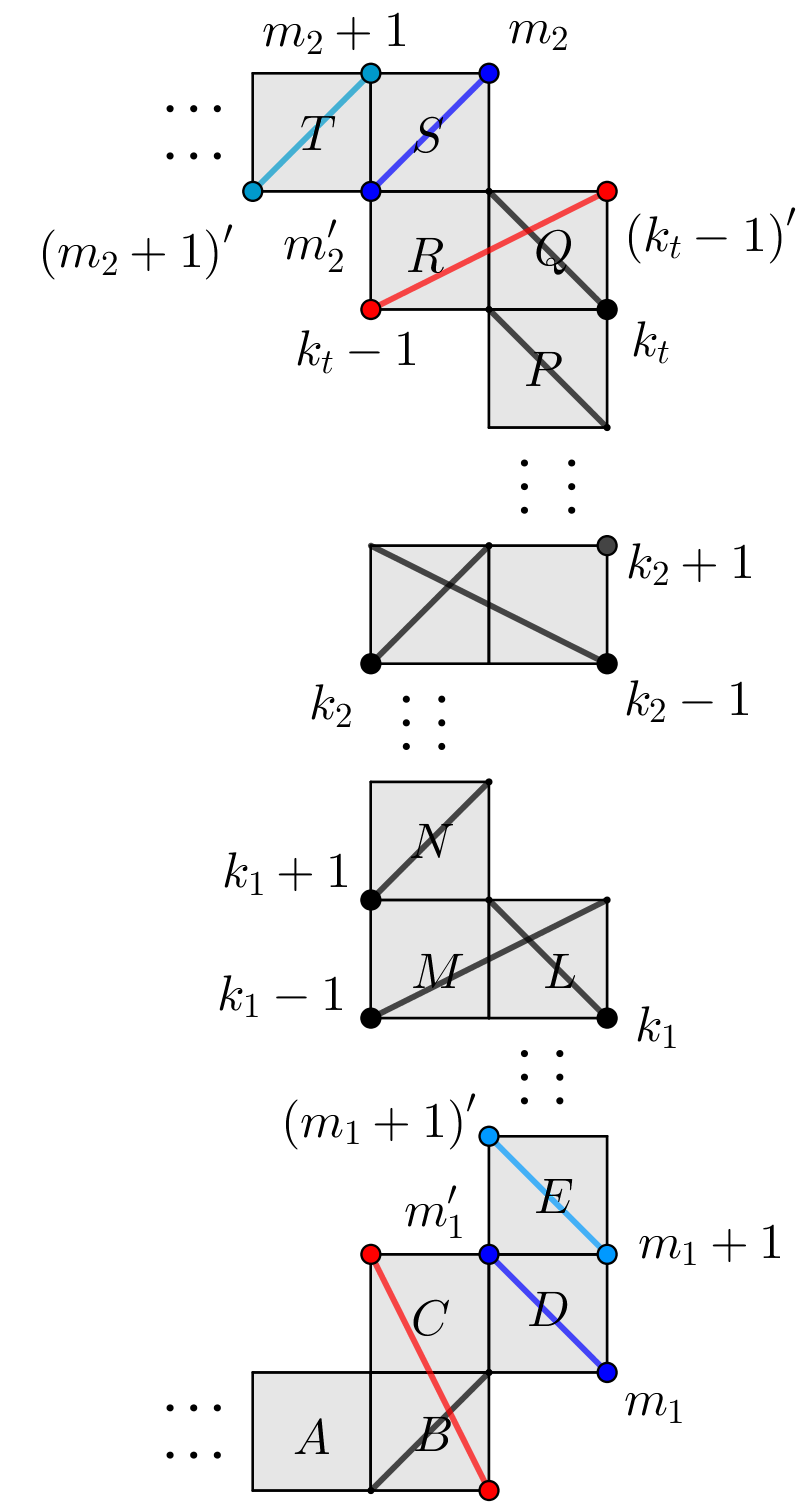}}\quad
		\subfloat[]{\includegraphics[scale=0.5]{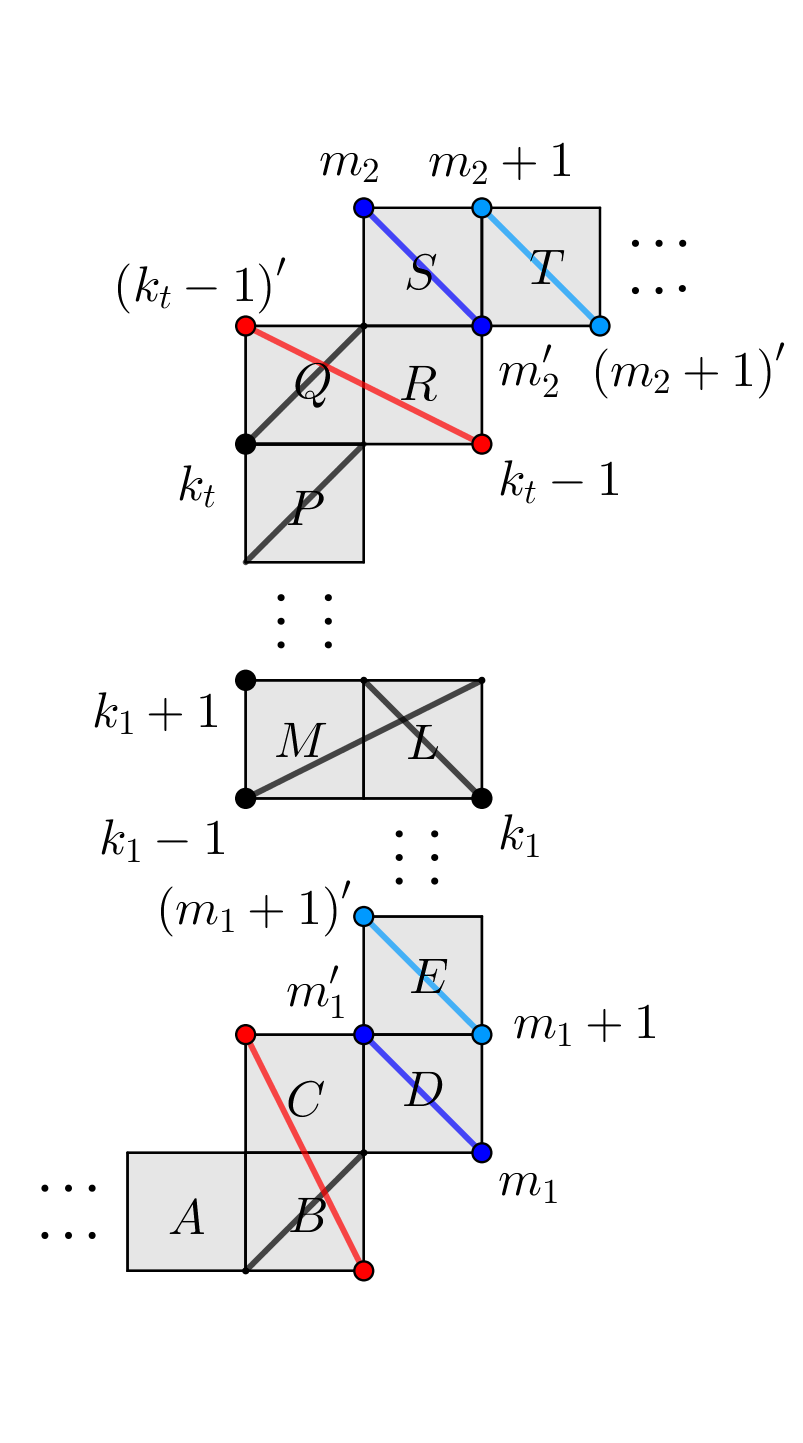}}\quad
		\subfloat[Either $E=P$ or $E\neq P$]{\includegraphics[scale=0.5]{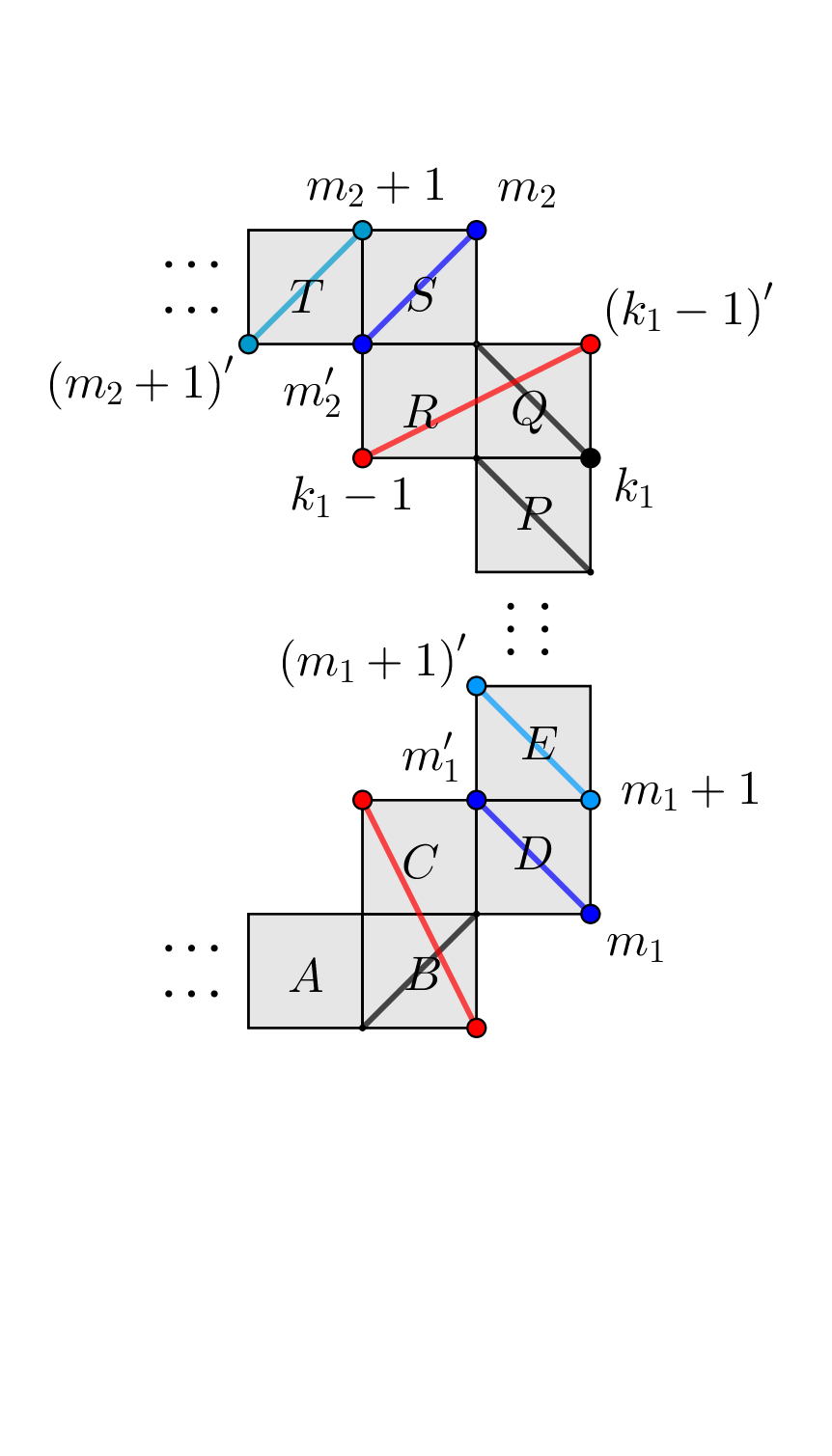}}
		\caption{Pieces of closed path with a zig-zag walk.}
		\label{Figure: example structure closed path zig-zag}
	\end{figure}
\end{remark}

\begin{proposition}\label{Prop: dimension + no cone point}
	Let $\cP$ be a closed path polyomino with a zig-zag walk, $\prec_{\cP}$ be a monomial order provided in Definition \ref{Defn: lex order} and $\Delta(\cP)$ be the simplicial complex attached to $\cP$ with respect to $\prec_{\cP}$. Then $\Delta(\cP)$ is a $(d-1)-$dimensional pure simplicial complex, where $d=\vert V(\cP)\vert/2$.
\end{proposition}

\begin{proof}
	We start proving that the dimension of $\Delta(\cP)$ is equal to $d-1$, where $d=\vert V(\cP)\vert/2$.  Denote by $F_0$ the set of the vertices of $\cP$ labelled by $Y_2$ (for instance, look at the orange points in Figure \ref{Figure: some vertices of F_0}). We firstly show that $F_0$ is a facet of $\Delta(\cP)$ with $d$ vertices. By Remark \ref{Remark: structure path with zig-zag walk} we can restrict ourself on one of the arrangements in Figure \ref{Figure: example structure closed path zig-zag}, up to rotations or reflections. We focus on Figure \ref{Figure: example structure closed path zig-zag} (a) (equivalently, refer to Figure \ref{Figure: some vertices of F_0} (a)), since the discussion for the other two situations is completely same. Denote the vertices of the cells that goes from $C$ to $R$ by $\cV(C,R)$. Observe that in $\cV(C,R)\cap F_0$ there are not two vertices of $\cP$, let say $v$ and $w$, such $\mathrm{in}(f)=x_vx_w$ for some generator of $I_{\cP}$. Hence $F_0$ is a face of $\Delta(\cP)$. Moreover, $F_0$ is a maximal face of $\Delta(\cP)$, since $F_0\cup \{i\}$ is not a face of $\Delta(\cP)$, for all $i\in [n]$. Then $F_0$ is a facet of $\Delta(\cP)$ with $d=\vert V(\cP)\vert/2$ vertices. In order to prove that the dimension of $\Delta(\cP)$ is $d-1$, we need to show that there is no face of $\Delta(\cP)$ which contains more than $d$ vertices. First, we know that $\vert \cP\vert = \vert V(\cP)\vert/2$, so this allows us to associate every cell of $\cP$ with a pair of vertices of $\cP$, in the following way: consider the path $\{D,E,\dots,L,M\}$ as shown in Figure \ref{Figure: example structure closed path zig-zag} (a)-(b), up to rotations or reflections (or $\{D,E,\dots,P,Q,R\}$ as in Figure \ref{Figure: example structure closed path zig-zag} (c)); we assign the pairs $\{m_1,m_1'\}, \{m_1+1,(m_1+1)'\},\dots,\{k_1,k_1'\}$ to $D,E,\dots,L$, respectively, and $\{k_1-1,(k_1-1)'\}$ to $M$; finally, this assignment extends to $\cP$, by treating $\cP$ as union of arrangements similar to $\{D,E,\dots,L,M\}$. Thus ensuring that each cell in $\cP$ is associated with exactly two distinct vertices in $V(\cP)$. Assume now by contradiction that there exists a face $H$ of $\Delta(\cP)$ with $\vert H\vert > \vert V(\cP)\vert/2=\vert \cP\vert$. Let $H_1=H\cap Y_1$ and $H_2=H\cap Y_2$, and set $h_1=\vert H_1\vert$ and $h_2=\vert H_2\vert$. For all $i\in H_1$, we have $i'\notin H_2$, due to $H$ be a face of $\Delta(\cP)$, so $H_1\cap \{i:i'\in H_2\}=\emptyset$. Since $h_1+h_2=\vert \cP\vert +1$, then from the previous observation it follows that there exists a $\{j,j'\}\subset H$, for some $j\in [n]$, which is a contradiction. Therefore, a facet with more that $\vert \cP\vert$ vertices cannot exist in $\Delta(\cP)$, so the dimension of $\Delta(\cP)$ is $d-1$, as desired. To finally prove the pureness, it is enough to show that there does not exist a facet of $\Delta(\cP)$ having less than $\vert \cP\vert$ vertices. Suppose that there is a facet $H$ of $\Delta(\cP)$ such that $\vert H\vert <\vert \cP\vert$. By employing a similar line of reasoning as previously used, it is easy to find a $j\in Y_1$ and $j'\in Y_2$ with $j,j'\notin H$ such that $H\cup \{j\}$ is a face of $\Delta(\cP)$, which is a contradiction with the maximality of $H$. Hence $\Delta(\cP)$ is a pure simplicial complex. 
\end{proof}

 If $\cP$ is a closed path without a zig-zag walk, then $\Delta(\cP)$ is shellable by \cite[Theorem 9.6.1]{Villareal}.  We aim to show that $\Delta(\cP)$ is shellable and, in particular, to provide a suitable shelling order, even in the case where $\cP$ has a zig-zag walk. Let us begin by introducing the following definitions of \textit{left}, \textit{right}, and \textit{upper step} of $\cP$.\\
 
 \begin{definition}\rm\label{Defn: steps}
 	Let $\cP$ be a closed path with a zig-zag walk and $\Delta(\cP)$ be the simplicial complex attached to $\cP$. Let $F$ be a facet of $\Delta(\cP)$ and $F'\subseteq F$ with $\vert F'\vert=3$. We do some rotations or reflections of $\cP$ in order to put the cell arrangements containing $F'$ as $\{D,E,\dots,P,Q,R\}$ in Figure \ref{Figure: some vertices of F_0}.
 	
 	\begin{itemize}
 	\item We say that $F'$ is a \textit{right step} of $F$ or that $F$ has a \textit{right step} $F'$ if only one of the following holds: 
 	\begin{itemize}
 	\item $F'=\{(a-1,b),(a,b),(a,b+1)\}$ for some $(a,b)\in V(\cP)$ and $(a,b)$ is the lower right corner of the cell $[(a-1,b),(a,b+1)]$ of $\cP$;
 	\item $F'=\{(a-2,b),(a,b),(a,b+1)\}$ for some $(a,b)\in V(\cP)$, $(a+1,b)\notin F$ and $(a,b)$ is the lower right corner of the cell $[(a-1,b),(a,b+1)]$ of $\cP$.
 	\end{itemize}
 	In such a case, $(a,b)$ is called the \textit{lower right corner} and $[(a-1,b),(a,b+1)]$ the \textit{step cell} of $F'$. \\
 	
 	\item Similarly, $F'$ is a \textit{left step} of $F$ or that $F$ has a \textit{right step} $F'$ if either
 		\begin{itemize}
 			\item $F'=\{(a,b+1),(a,b),(a+1,b)\}$ for some $(a,b)\in V(\cP)$ and $(a,b)$ is the lower left corner of the cell $[(a,b),(a+1,b+1)]$  of $\cP$, or
 			\item $F'=\{(a,b+1),(a,b),(a+2,b)\}$ for some $(a,b)\in V(\cP)$, $(a+1,b)\notin F$ and $(a,b)$ is the lower left corner of the cell $[(a,b),(a+1,b+1)]$  of $\cP$.
 		\end{itemize}
 		Here, $(a,b)$ is said to be the \textit{lower left corner} and $[(a,b),(a+1,b+1)]$ the \textit{step cell} of $F'$.\\
 		
 		\item We say that $F'$ is an \textit{upper step} if, with reference to Figure \ref{Figure: some vertices of F_0}, $F'=\{m_2+1,m_2,k\}$ for some $k$ in the vertical edge interval $[(j-1)',k_t']$ and every vertex in $]k,k_t'[$ does not belong to $F$. \\
 		In this case, $m_2$ and $S$ are called the \textit{upper corner} and the \textit{step cell} of $F'$, respectively.
 		\end{itemize}
  \end{definition}

Observe that $\{m_2+1,m_2,k_t'\}$ or $\{m_2+1,m_2,(k_t-2)'\}$ can be viewed as either right steps or upper steps, depending on the perspective from which we consider the cell arrangement. However, this distinction does not affect our arguments. Now, we are ready to discuss how the vertices of a facet of $\Delta(\cP)$ can be arranged in $\cP$. \\

\begin{discussion}\rm \label{Discussion}
	Let $\cP$ be a closed path polyomino with a zig-zag walk, $\prec_{\cP}$ be a monomial order provided in Definition \ref{Defn: lex order} and $\Delta(\cP)$ be the simplicial complex attached to $\cP$ with respect to $\prec_{\cP}$. In this discussion we want to show how a facet of $\Delta(\cP)$ can be figure out in $\cP$. We will provide an explanation that avoids extreme formalism, making it easier to understand the process. Recall that $F_0$ is the set of the vertices of $\cP$ labelled by $Y_2$. We know that $\cP$ consists of cell arrangements $\{D,E,\dots, L,M\}$ up to rotations and reflections, so we focus on the sequence of cells $\{D,E,\dots,L,M,\dots,K,O\}$ referring to Figure \ref{Figure: some vertices of F_0} (a) and, in particular, let us start restricting ourself just on $\{D,E,\dots,L,M\}$. The discussion is the same if we consider Figure \ref{Figure: some vertices of F_0} (b).
	
	\begin{figure}[h]
		\centering
		\subfloat[$m_1<k_1<\dots<k_t<m_2$]{\includegraphics[scale=0.6]{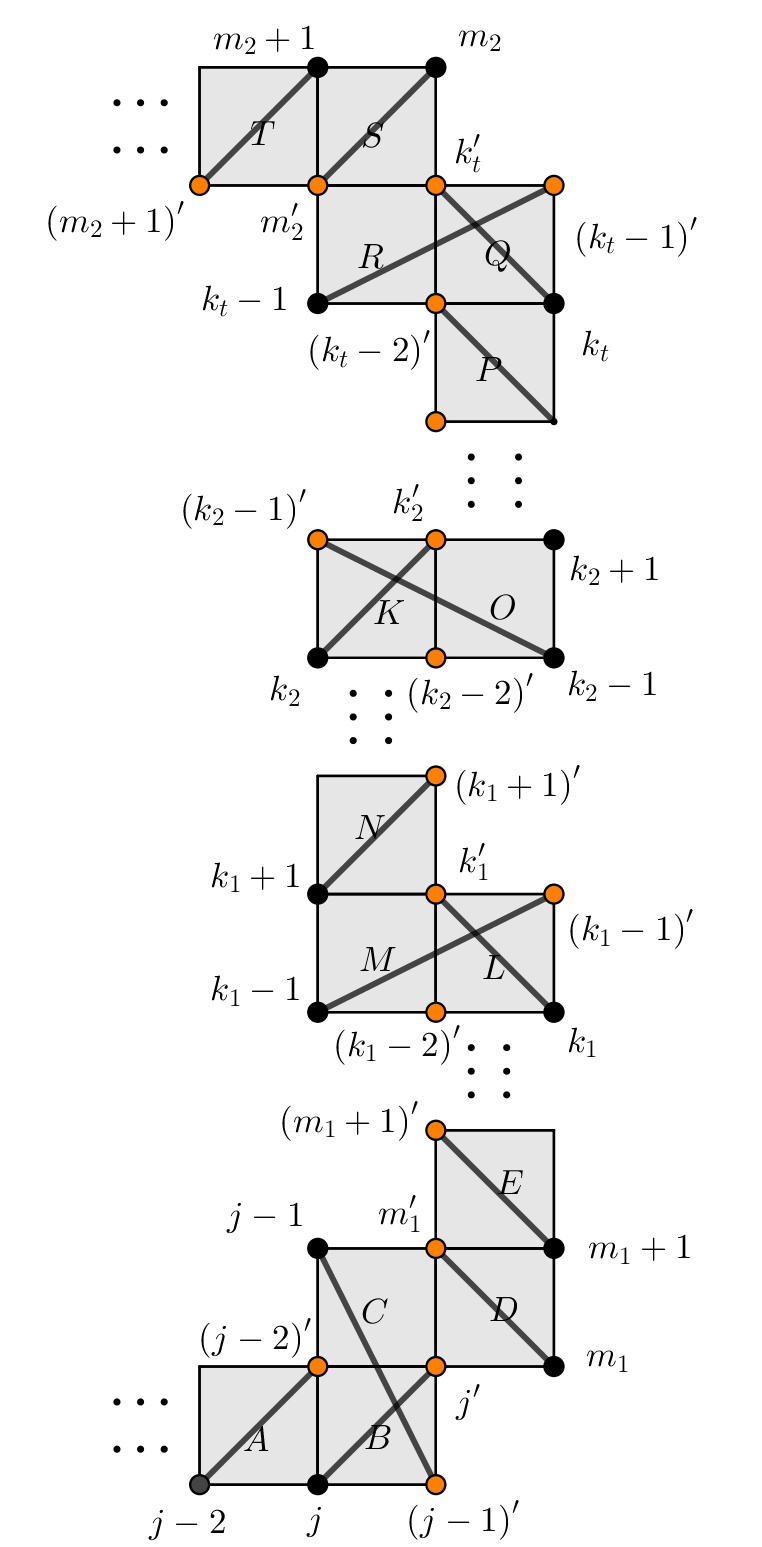}}\
		\subfloat[$m_1<k_1<\dots<k_t<m_2$]{\includegraphics[scale=0.6]{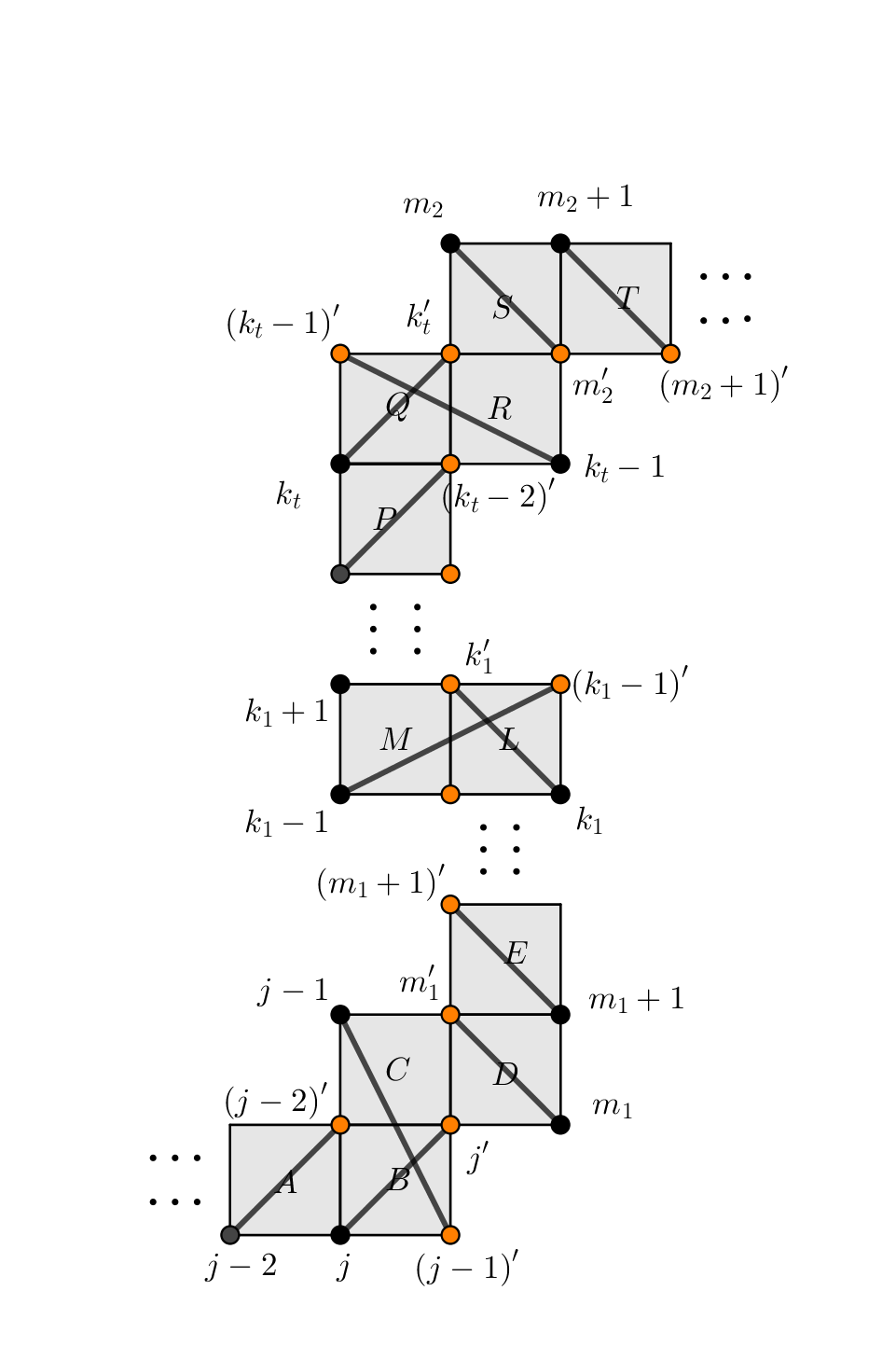}}
		\caption{Some vertices of $F_0$.}
		\label{Figure: some vertices of F_0}
	\end{figure}

	\begin{enumerate}[1)]
		\item Consider the vertex $k_1'$ and observe that $F_1=\left(F_0\setminus\{k_1'\}\right)\cup \{k_1\}$ is a facet of $\Delta(\cP)$ and $\{(k_1-1)',k_1,(k_1-2)'\}$ is a right step of $F_1$. Next, take the vertex $(k_1-2)'$ and the facet $F_1$, so $F_2=\left(F_1\setminus\{(k_1-2)'\}\right)\cup \{(k_1-2)'\}$ is also a facet of $\Delta(\cP)$ and $\{(k_1-3)',k_2,k_1\}$ is a right step of $F_2$. We can continue this procedure until we reach the vertex $m_1'$, obtaining a new facet from the previous one by replacing just one vertex, specifically, for all $j=2,\dots, k_1-m_1$ the set $F_j=\left( F_{j-1}\setminus \{(k_1-j)'\}\right)\cup \{k_1-j\} $ is a facet of $\Delta(\cP)$ and $k_1-j$ is a lower right corner of a step of $F_j$. 
		\item Now, consider $F_0$ again. If we take $\left(F_0\setminus\{(k_1-1)'\}\right)\cup\{k_1-1\}$, then we do not obtain a face of $\Delta(\cP)$ because $\{k_1-1,k_1'\}$ is contained in $ \left(F_0\setminus\{(k_1-1)'\}\right)\cup\{k_1-1\}$ and $\mathrm{in}(f)=x_{k_1'}x_{k_1-1}$, where $f=x_{k_1'}x_{k_1-1}-x_{k_1+1}x_{(k_1-2)'}$. A similar contradiction arises if we replace $k_1'$ with $k_1-1$ in $F_0$, or if we replace $k_1'$ with $k_1-1$ (or any vertex in $[k_1-1,k_2]$) in $F_0$. This also applies to any replacement of $(k_1-j)'$  with $k_1-i$ in $F_{j-1}$ for all $j=2,\dots, k_1-m_1$ and for all $i=j+1,\dots,j-1$ (if $k_1-i$ exists). However, intuitively, if we shift every orange vertex $v$ in the interval $[k_1',k_2']$ from the top to the bottom in the related opposite $v'$, then we will eventually move $(k_1'-1)'$ to $k_1-1$. Formally, consider $F_0$ and the vertex $k_2'$, then $G_1=\left(F_0\setminus\{k_2'\}\right)\cup \{k_2\}$ is a facet of $\Delta(\cP)$ and $\{(k_2-1)',k_2,(k_2-2)'\}$ is a left step of $G_1$. As done in 1), continue this procedure until reaching the vertex $k_1'$, consistently obtaining a new facet from the previous one by replacing only one orange vertex at a time; that is, $G_j=\left( G_{j-1}\setminus \{(k_2-j)'\}\right)\cup \{k_2-j\}$ with $k_2-j$ a lower left corner of a step of $G_j$, for all $j=2,\dots, k_2-k_1$. We now need to distinguish two situations. If $j<k_2-k_1$, then $G_j$ has the left step $\{k_2-j+1,k_2-j,(k_2-j-1)'\}$. Moreover, we cannot replace $(k_1-1)'$ with $k_1-1$ in $G_j$ but we can do it for $k_1'$ with $k_1$. Thus, take $G_j$, replace $k_1'$  with $k_1$ in $G_j$ and, then,  apply the procedure described in 1), where $(G_j\setminus\{k_1'\})\cup \{k_1\}$ takes the place of $F_0$. Assume $j=k_2-k_1$, so $G_j$ has the left step $\{k_1+2,k_1+1,k_1'\}$. Then, there are two possibilities:
		\begin{enumerate}[(a)]
		\item we can apply procedure 1) to $G_j$, where $G_0$ itself plays the role of $F_0$; thus every new facet will be the left step $\{k_1+2,k_1+1,k_1'\}$ and a right step with lower right corner in $[m_1,k_1]$. 
		\item Alternatively, we can replace $(k_1-1)'$ with $k_1-1$ in $G_j$. This replacement is now feasible, that is, $G'= (G_j\setminus\{(k_1-1)'\})\cup \{k_1-1\}$ is a facet of $\Delta(\cP)$ with left step $\{k_1+1,k_1-1,(k_1-2)'\}$.Then, apply procedure 1) to $G'$, so that every new facet will be the left step $\{k_1+1,k_1-1,k_1\}$ and a right step with lower right corner in $[m_1,k_1[$. \end{enumerate} 
		In both cases, a facet of $\Delta(\cP)$ can be obtained from the previous one by replacing just one orange vertex.
		
		\item Consider $F_0$ and $\{P,Q,R\}$. As we said before, the set $(F_0\setminus \{(k_t-1)'\})\cup\{k_t-1\}$ is not a facet of $\Delta(\cP)$, meaning that we cannot obtain a facet of $\Delta(\cP)$ by replacing $(k_t-1)'$ with $k_t-1$ in $F_0$. However, if we first replace $k_t'$ with $k_t$ in $F_0$, then that previous replacement becomes possible. Hence, consider $H=(F_0\setminus \{k_t'\})\cup\{k_t\}$ (so $H$ has $\{(k_t-2)',k_t,(k_t-1)'\}$ as a right step) and we can proceed in two different ways. 
		\begin{enumerate}[(a)]
			\item  Apply procedure 1) to $H$ until reaching $k_{t-1}'$.
			\item Define $H'=(H\setminus \{(k_t-1)'\})\cup\{k_t-1\}$. Here, $H'$ has $\{(k_t-2)',k_t-1,m_2'\}$ as a left step. Then, apply procedure 1) to $H'$ until reaching $k_{t-1}'$ again. The faces, which we obtain, have the left step $\{k_t,k_t-1,m_2'\}$ and a right step with right lower corner in $[k_t-1,k_t[$.
		\end{enumerate}
		 Thus, in both approaches, we get a facet of  $\Delta(\cP)$ from the previous one by replacing just one orange vertex.
		\item Finally, we analyze the scenario involving the set $\{P,Q,R,S,T\}$ or equivalently $\{A,B,C,D,E\}$, since
		$\cP$ can be constructed by rotating or reflecting one of the cell configurations in Figure \ref{Figure: some vertices of F_0} overlapping $\{A,B,C,D,E\}$ with $\{P,Q,R,S,T\}$. For simplicity, we look at Figure \ref{Figure: some vertices of F_0} (a). Consider a facet $F$ of $\Delta(\cP)$. If $m_2\in F$, (the vertices $m_2+1,\dots$ are also in $F$), then it is not possible to replace $(k_t-1)'$ with $k_t-1$ in $F$ to obtain a facet of $\Delta(\cP)$. However, it is possible to replace $k_t'$ with $k_t$, that is, $K=(F\setminus\{k_t'\})\cup\{(k_t-1)'\}$ is a facet of $\Delta(\cP)$ having the right step $\{(k_t-2)',k_t,(k_t-1)'\}$ (note that $\{m_2+1,m_2,(k_t-2)'\}$ is a right step of $K$ as well). Now, by applying 1) at $K$, we obtain a facet $K'$ having a right step whose lower right corner is in $[k_2-1,k_t[$. Moreover, in this case, it is straightforward to verify that there exists a vertex $v$ in the edge interval $[(j-1)',(k_t-2)'[$ such that $\{v,m_2,m_2+1\}$ is a right step of $K'$. \\
	\end{enumerate}
	
	The procedure described in the previous four points can be naturally extended to any sequence such as  $\{D,E,\dots,L,M,\dots,N,K\}$ within $\{D,E,\dots,P,Q,R\}$ (up to rotations and reflections), and consequently to the entire sequence $\{D,E,\dots,P,Q,R\}$. Additionally, since $\cP$ consists of disjoint union of the arrangements $\{D,E,\dots,P,Q,R\}$ given in Figure \ref{Figure: some vertices of F_0}, the discussed procedure can be applied piece by piece to the whole of $\cP$. \\
\end{discussion}

	\begin{example}\rm
	We provide an example to illustrate how to identify a facet of $\Delta(\cP)$ in $\cP$. Consider the closed path $\cP$ in Figure \ref{Figure: exa of a facet} (a); the orange circle vertices and the black cross ones represent the facet $F_0$ and another facet $F$ of $\Delta(\cP)$, respectively. Observe that $F$ has $\{6',7,8\}$, $\{18',20,19'\}$, $\{24',26,25'\}$, $\{24',1,2\}$ as right steps ($\{24',1,2\}$ can be viewed as upper step, too) and $\{12,11,13'\}$, $\{14',15,17'\}$, $\{19',21,23\}$ as left steps. Take in consideration the facet $G$ in Figure \ref{Figure: exa of a facet} (b), then $\{5,6,25'\}$ and $\{1,2,18'\}$ are upper steps, $\{17',15,14'\}$ is a left step and $\{18',20,19'\}$, $\{22',24,26\}$ are right steps. Finally, observe that $F_0$ never has left, right or upper steps.  
\begin{figure}[h]
	\subfloat[Facet $F$.]{\includegraphics[scale=0.65]{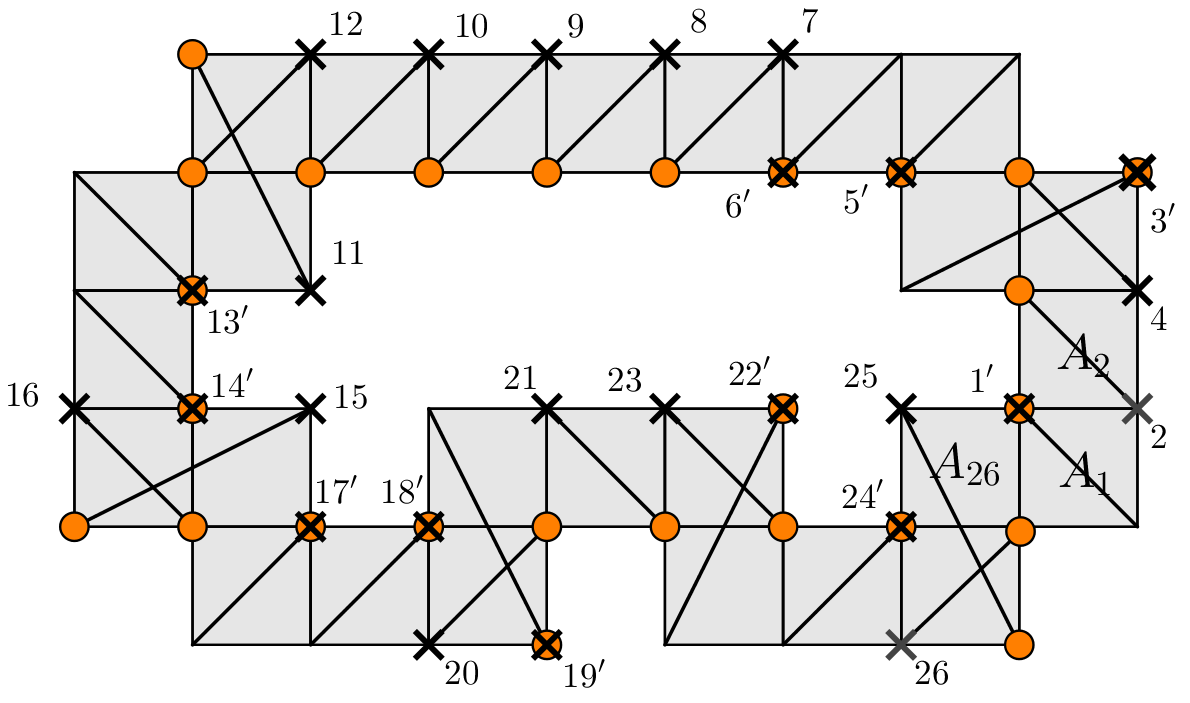}}
	\subfloat[Facet $G$.]{\includegraphics[scale=0.65]{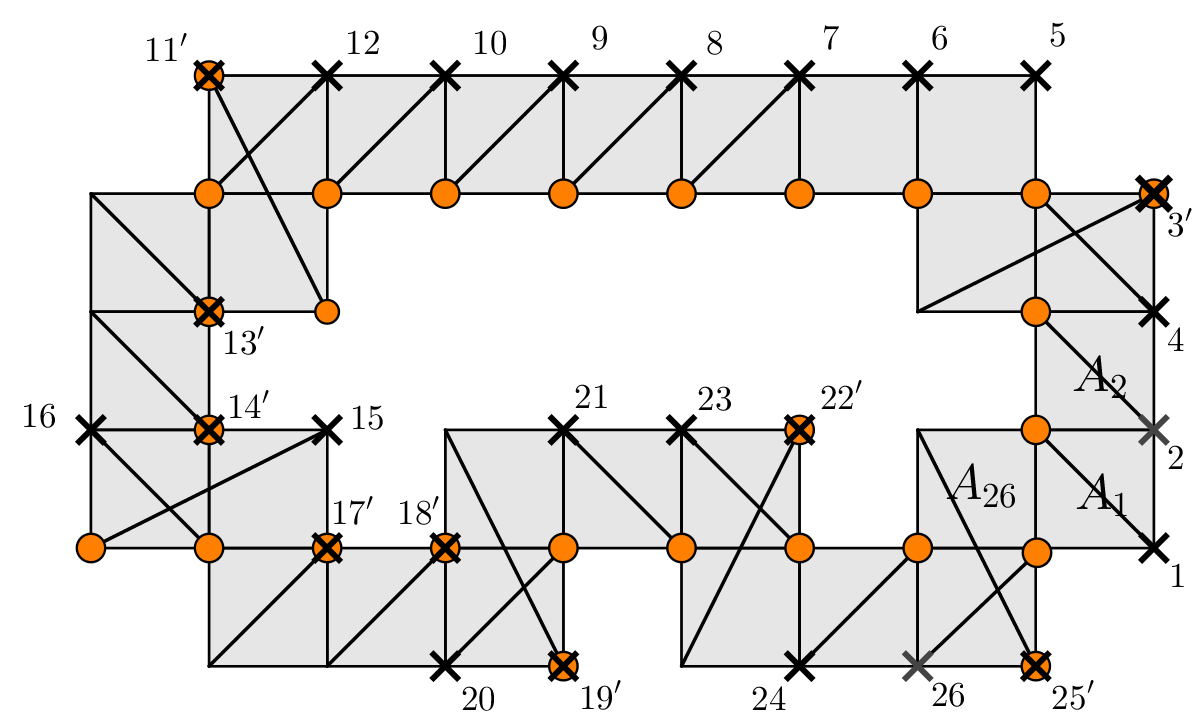}}\
	\caption{Example of facets.}
	\label{Figure: exa of a facet}
\end{figure}
	\end{example}

In what follows, we will show that $\Delta(\cP)$ is shellable. We will provide an explicit shelling order in the next definition, explained through a pseudo-algorithm. The used commands come primarily from \texttt{Macaulay2} (refer to \cite{M2} for further details).\\

\begin{definition}\rm \label{Shelling order}
	Let $\cP:A_1,\dots,A_n$ be a closed path polyomino with a zig-zag walk, where $A_i\neq A_j$ for all $i,j=1,\dots,n$ and $i\neq j$, $\prec_{\cP}$ be a monomial order provided in Definition \ref{Defn: lex order} and $\Delta(\cP)$ be the simplicial complex attached to $\cP$ with respect to $\prec_{\cP}$. We want to define a linear order of the facets of $\Delta(\cP)$ in a recursive manner.
	By considering $A_1,A_2,\dots$ and by walking on the path $\cP$ from $A_1,A_2$ to $A_n$, we denote by $\cC_1,\dots,\cC_s$ the cell arrangements of $\cP$ as in Figure $\ref{Figure: some vertices of F_0}$ (up to rotations and reflections). \\

\textit{\underline{Procedure 1.}} Let us start with $\cC_1$ and consider the facet $F_0$. Referring to Figure \ref{Figure: some vertices of F_0}, we identify $D=A_1$, $E=A_2$ and $C=A_{n}$. We will outline two steps in this first procedure.\\

\noindent\textit{First step.} Refer to Figure $\ref{Figure: some vertices of F_0}$ (a), consider $\{O,P,\dots,Q,R\}$ and set $a=k_t$ and $b=k_{t-1}$. \\

\begin{enumerate}[(1)]
	\item $G_0=F_0$;\\
	$F_1=(F_0\setminus\{a'\})\cup \{b\}$;\\
	$F=F_1$;\\
	$L=\mathrm{toList}\{G_0,F_1\}$;\\
	\hspace*{0.3cm}FOR $i$ from $a-2$ to $b+1$ in descending order DO( \\
	\hspace*{0.5cm}$F=(F\setminus \{i\})\cup\{i'\} $;\\
	\hspace*{0.5cm}$L=\mathrm{append}(L,\{F\})$;\\
	\hspace*{0.3cm});\\
	Denote by $G_1$ the last facet in the list $L$;\\
	
	\item $F=(F_1\setminus\{(a-1)'\})\cup\{a-1\}$;\\
	$L=\mathrm{append}(L,\{F\})$;\\
	\hspace*{0.3cm}FOR $i$ from $a-2$ to $b+1$ in descending order DO(\\
	\hspace*{0.5cm}$F=(F\setminus \{i\})\cup\{i'\} $;\\
	\hspace*{0.5cm}$L=\mathrm{append}(L,\{F\})$;\\
	\hspace*{0.3cm});\\
	Denote by $G_2$ the last facet in the list $L$;\\
\end{enumerate}

		\noindent\textit{Second Step.} Now, consider an arrangement such as $\{N,\dots,K,O,\dots,R\}$ in Figure \ref{Figure: some vertices of F_0} (a), if it exists. \\

		\noindent	FOR $j$ from $0$ to $\vert L\vert$ DO(\\
		\hspace*{0.3cm}$H_j$ is the $j$-th facet in the list $L$;\\
		\hspace*{0.5cm}IF $H_j\neq G_1,G_2$ THEN\\
		\hspace*{0.7cm}Apply \textit{First Step} (1) replacing $F_0$ with $H_j$ and setting $a=k_{t-1}$ and $b=k_{t-2}$. Hence we obtain a list $L'$.\\
		\hspace*{0.7cm}$L=\mathrm{join}(L,L')$;\\
		\hspace*{0.5cm}ELSE Apply \textit{First Step} (1) and (2) replacing $F_0$ with $H_j$ and setting $a=k_{t-1}$ and $b=k_{t-2}$. Hence we obtain a list $L'$.\\
		\hspace*{0.7cm}$L=\mathrm{join}(L,L')$;\\
		\hspace*{0.5cm});\\

		\noindent Therefore, we obtain a new list $L$ from \textit{Second Step} (look at Example \ref{Example}). Now, for all facet in $L$, we can apply the \textit{First Step} to an arrangement (if it exists) as a reflected $\{N,\dots,K,O,\dots,R\}$ in Figure \ref{Figure: some vertices of F_0} (a) and $a=k_{t-2}$ and $b=k_{t-3}$. Finally, we get a new list $L$ of all facets of $\Delta(\cP)$ that have a left or right step in $\{D,E\dots,Q,R\}$. This concludes \textit{Procedure 1.}\\

\textit{\underline{Extension to further arrangements.}} Consider $\cC_2$ and perform suitable rotations or reflections of $\cP$ so that $\cC_2$ is positioned as in Figure \ref{Figure: some vertices of F_0}. \\

	\noindent FOR $i$ from $0$ to $\vert L\vert$ DO(\\
	Denote by $H_i$ the $i$-th facet in $L$;\\
	\hspace*{0.3cm}IF $H_i$ contains $j-1$ as in Figure \ref{Figure: some vertices of F_0} THEN(\\
	\hspace*{0.3cm}Apply \textit{Procedure 1.} in $\cC_2$ to $H_i$;\\
	\hspace*{0.3cm}We get a list $L'$, where some sets are not facets because they have $j-1$ and $m_1$;\\
	\hspace*{0.3cm}Denote by $H_i^{(1)},\dots, H_i^{(l)}$ the sets in $L'$ which contain $m_1$;\\
	\hspace*{0.6cm} FOR $k$ from $0$ to $l$ DO(\\
    \hspace*{0.6cm} $L'=\mathrm{delete}(H_i^{(k)}, L')$;\\
    \hspace*{0.6cm} );\\ 
    \hspace*{0.3cm}$L=\mathrm{join}(L,L');$\\
    \hspace*{0.3cm});\\
    
    \noindent \hspace*{0.3cm}IF $H_i$ does not contain $j-1$ as in Figure \ref{Figure: some vertices of F_0} THEN(\\
    \hspace*{0.3cm}Apply \textit{Procedure 1.} in $\cC_2$ to $H_i$;\\
    \hspace*{0.3cm}We get a list $L'$ of facets;\\
    \hspace*{0.5cm});\\ 
    $L=\mathrm{join}(L,L');$\\
    );\\
	
\noindent Finally, we get a new list $L$ of the facets of $\Delta(\cP)$ that have a left, a right or an upper step in $\cC_2$. \\

\textit{\underline{Iterative process for remaining arrangements.}} This process can be repeated for $\cC_3$, meaning that we consider the list $L$ and apply the previously described procedure to each facet of $L$ on $\cC_3$. This continues until $\cC_s$, resulting in a new list $L$ of sets that are not all facets of $\Delta(\cP)$. In particular, consider the common part of $\cC_s$ and $\cC_1$, especially the sequence ${A_{n-1}, A_n, A_1}$. If we refer to Figure \ref{Figure: some vertices of F_0}, we have $B = A_{n-1}$, $C = A_{n}$, $D = A_1$, and $j-1$ and $m_1$ labelled by $n-1$ and $1$, respectively. There are some sets in $L$ that simultaneously contain $1$ and $n-1$. We simply need to remove these sets, obtaining a list $L$ which contains all the facets of $\Delta(\cP)$.\\ 
\end{definition}

\begin{example}\rm\label{Example}
	Here, we give an example of the order $<$ on $\cF(\cP)$ that we define earlier. Let $\cP$ be the closed path in Figure \ref{Figure: Example order facets} and let $F_0=\{1',\dots,26'\}$, indicated by the orange vertices. We denote $\cC_1=\{A_{24},A_{25},A_{26},A_1,\dots,A_6\}$, $\cC_2=\{A_2,\dots,A_{14}\}$, $\cC_3=\{A_{10},\dots,A_{18}\}$ and $\cC_4=\{A_{14},\dots,A_{26},A_1,A_2\}$.
	
	\begin{figure}[h]
		\centering
		\includegraphics[scale=0.65]{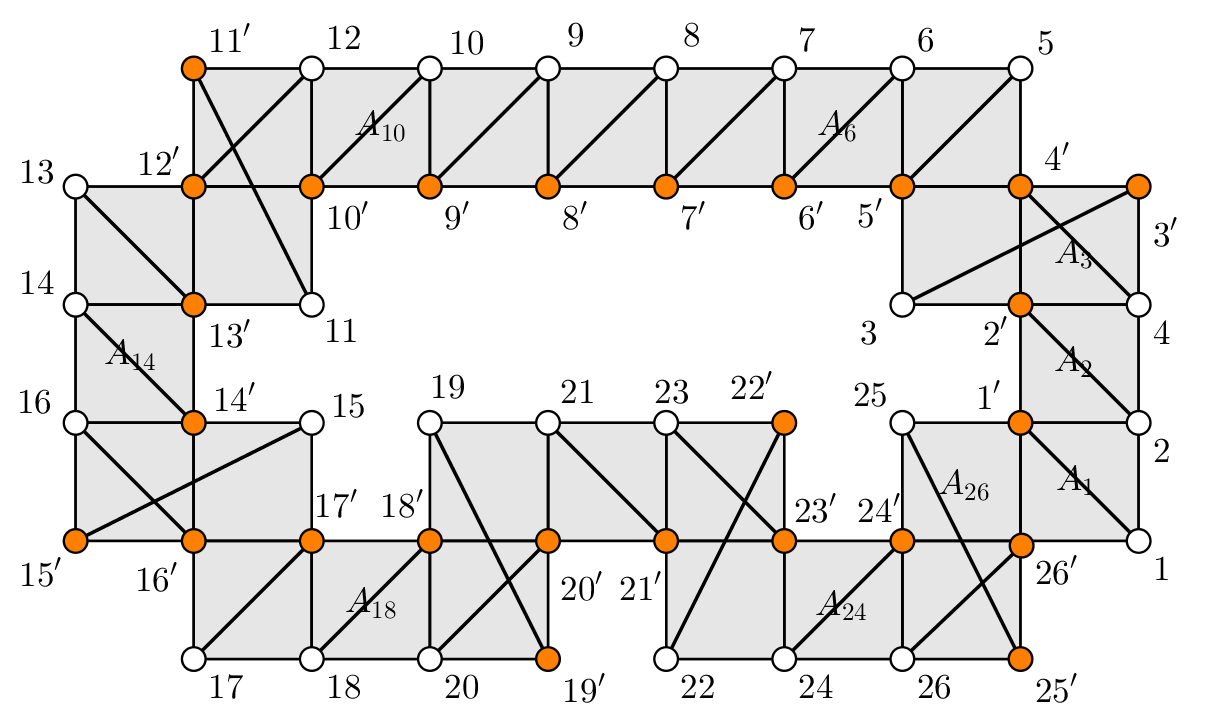}
		\caption{A closed path with zig-zag walks.}
		\label{Figure: Example order facets}
	\end{figure}

	 \noindent Applying Definition \ref{Defn: lex order}, we get the following facets: 
	\begin{itemize}
	\item $F_1=\{1',2',3',4,5',\dots,26'\}$ with the right step $\{2',4,3'\}$,  $F_2=\{1',2,3',4,5',\dots,26'\}$ with the right step $\{1',2,4\}$ and  $F_3=\{1,2,3',4,5',\dots,26'\}$ with the right step $\{26',1,2\}$ in the First Step; 
	\item $F_4=\{1',2',3,4,5',\dots,26'\}$  with the left step $\{2',3,5'\}$, $F_5=\{1',2,3,4,5',\dots,26'\}$ with the left step $\{2',3,5'\}$ and right step $\{1',2,4\}$, $F_6=\{1,2,3,4,5',\dots,26'\}$ with the left step $\{2',3,5'\}$ and right step $\{26',1,2\}$, in the Second Step.	
	\end{itemize}
 Therefore, $L=(F_0,\dots,F_6)$. Now, consider $\cC_2$ and for any facet in $L$ we apply the previous procedure on $\cC_2$. That is,
 \begin{itemize}
	\item from $F_0$, we get: $F_7=\{1',\dots,10',11',12,13',\dots,26'\}$ with the right step $\{10',12,11'\}$, $F_8=\{1',\dots,9',10,11',12,13',\dots,26'\}$ with the right step $\{9',10,12\}$, $F_{9}=\{1',\dots,8',9,10,11',12,13',\dots,26'\}$ with the right step $\{8',9,10\}$, and so on, up to $ F_{13}=\{1',\dots,4',5,6,7,8,9,10,11',12,13',\dots,26'\}$ with the right step $\{4',5,6\}$. 
	\item Taking $F_1$, we have: $F_{14}=\{1',2',3',4,5'\dots,10',11',12,13',\dots,26'\}$ with the right steps $\{2',4,3'\}$ and $\{10',12,11'\}$, $F_{15}=\{1',2',3',4,5',\dots,9',10,11',12,13',\dots,26'\}$ with the right steps $\{2',4,3'\}$ and $\{9',10,12\}$, $F_{16}=\{1',2',3',4,5',\dots,8',9,10,11',12,13',\dots,26'\}$ with the right steps $\{2',4,3'\}$ and $\{8',9,10\}$, and so on, up to, $F_{20}=\{1',2',3',4,\dots,10,11',12,13',\dots,26'\}$ with the right steps $\{2',4,3'\}$ and $\{4',5,6\}$. 
	
	\item Considering $F_2$, we have: $F_{21}=\{1',2,3',4,5'\dots,10',11',12,13',\dots,26'\}$ with the right steps $\{1',2,4\}$ and $\{10',12,11'\}$, $F_{22}=\{1',2,3',4,5'\dots,8',9,10,11',12,13',\dots,26'\}$ with the right steps $\{1',2,4\}$ and $\{8',9,10\}$, and so on, up to, $F_{27}=\{1',2,3',4,\dots,10,11',12,13',\dots,26'\}$ with the right step $\{1',2,4\}$ and the upper step $\{1',5,6\}$. 
	\item We repeat this argument for $F_3$ and $F_4$; in particular, from $F_4$, we obtain: $F_{35}=\{1',2',3,4,5'\dots,10',11',12,13',\dots,26'\}$  with the right step $\{10',12,11'\}$ and the left step $\{2',3,5'\}$, $F_{36}=\{1',2',3,4,5',\dots,9',10,11',12,13',\dots,26'\}$ with the right step $\{10',12,11'\}$ and the left step $\{2',3,5'\}$,
	$F_{37}=\{1',2',3,4,5',\dots,8',9,10,11',12,13',\dots,26'\}$ with the right step $\{8',9,10'\}$ and the left step $\{2',3,5'\}$,
	and so on, up to,
	$F_{41}=\{1',2',3,4,5,\dots,10,11',12,13',\dots,26'\}$, but $F_{41}$ is not a facet because $\{3,5\}\subset F_{41}$, so we do not include it in the list. Similar arguments can apply for $F_5$, $F_6$ and $F_7$. 
 \end{itemize}
Therefore, we obtain a list $L=(F_0,\dots,F_{58})$ of facets of $\Delta(\cP)$. Now, consider $\cC_3$ and rotate it in order that $\cC_3$ is positioned as in Figure \ref{Figure: some vertices of F_0} (a). For all facet in $L$ from $F_0$ to $F_{58}$ we apply the previous arguments, obtaining a new list $L$ with $p$ facets, for some $p\in \NN$. Finally, repeat the same process for $\cC_4$ and for any facets in $L$. We now point out the procedure described in the last paragraph of \ref{Defn: lex order}. For instance, consider $F_3$ and we have to replace $26'$ with $26$. Then, for some $q>p$, we get $F_q=\{1,2,3',4,5',\dots,25',26\}$ with the upper step $\{24',1,2\}$ and right step $\{24',26,25'\}$, $F_{q+1}=\{1,2,3',4,5',\dots,23',24,25',26\}$ with the upper step $\{23',1,2\}$ and $\{23',24,26\}$; continuing, we have $F_{q+2}=\{1,2,3',4,5',\dots,24',25,26\}$ which is not a facet since $1,25\in F_{q+2}$ so we do not include it in the list. In the end, we obtain the desired list $L$.\\	
\end{example}

\begin{definition}\label{Defn: facets compare}\rm 
	Let $\cP$ be a closed path polyomino with a zig-zag walk, $\prec_{\cP}$ be a monomial order provided in Definition \ref{Defn: lex order}, $\Delta(\cP)$ be the simplicial complex attached to $\cP$ with respect to $\prec_{\cP}$ and $L$ be the order list of the facets of $\Delta(\cP)$ given in Definition \ref{Defn: lex order}. Let $F,G$ be two facet of $\Delta(\cP)$. Then there exist two indices $i,j$ with $i\neq j$ such that $F$ and $G$ are the $i$-th and the $j$-th facets in the list $L$, respectively. We say that $F>G$ (or $G<F$) if $i>j$.\\
\end{definition}

	For example, with reference to Figure \ref{Figure: exa of a facet}, it is easy to verify that $F>G$, since $F$ and $G$ are respectively defined starting from $F_2$ and $F_3$ in the list $L$ provided in Example \ref{Example}. \\

From now, the word \textit{step} includes the lower right, the lower left and the upper ones. Moreover, for simplicity, a vertex which is a right lower corner (or similarly, a lower-left or upper corner) of a step of a facet of $\Delta(\cP)$ is referred to as a \textit{corner of a step}. Keeping these notation in mind, we state the following result.\\

\begin{theorem}\label{Thm: shelling}
		Let $\cP$ be a closed path, $\prec_{\cP}$ be a monomial order provided in Definition \ref{Defn: lex order}, $\Delta(\cP)$ be the simplicial complex attached to $\cP$ with respect to $\prec_{\cP}$. Then $\Delta(\cP)$ is shellable. \\
		Suppose that $\cP$ has a zig-zag walk. Let $L$ be the order list of the facets of $\Delta(\cP)$ given in Definition \ref{Defn: lex order} and $\vert L\vert=l$. Then $\langle F_0,\dots,F_{i-1} \rangle \cap \langle F_i\rangle$ is generated by $$ \big\{F_i\setminus\{v\} : v\ \text{is a corner of a step of}\ F_i\big\},$$
		for all $i=1,\dots, l$.
\end{theorem}

\begin{proof}
	If $\cP$ has no zig-zag walks, then from \cite[Theorem 6.2]{Cisto_Navarra_closed_path} and Proposition \ref{Prop: Grobner basis} we have that $I_{\cP}$ is a toric ideal whose initial ideal with respect to $\prec_{\cP}$ is squarefree, so $\Delta_{\cP}$ is a shellable from \cite[Theorem 9.6.1]{Villareal}.\\
	Suppose that $\cP$ has a zig-zag walk. Let $i,j\in \{1,\dots,l\}$ with $j<i$. From Discussion \ref{Discussion} and Definitions \ref{Defn: lex order} and \ref{Defn: steps} we have that there exist either a lower right or a lower left corner or an upper corner, let us say $w$, of a step of $F_i$ such that $w\in F_i\setminus F_j$ and an integer $k<i$ such that $F_{i}\setminus F_k=\{w\}$. In conclusion, looking at \cite[Definition 5.1.11]{Bruns_Herzog}, we have that $\Delta_{\cP}$ is shellable and $\langle F_0,\dots,F_{i-1} \rangle \cap \langle F_i\rangle$ is generated by the faces $F_i\setminus\{v\}$, where $v$ is either the lower right or the lower left or the upper corner of a step of $F_i$.
\end{proof}

\begin{corollary}\label{Coro: Cohen-Macaulay}
	Let $\cP$ be a closed path with a zig-zag walk. Then $K[\cP]$ is a Cohen-Macaulay ring and $\dim(K[\cP])=\vert \cP\vert$.  
\end{corollary}

\begin{proof}
	From Proposition \ref{Prop: dimension + no cone point} and Theorem \ref{Thm: shelling}, we have that $K[\Delta(\cP)]$ is a $(d-1)-$dimensional shellable simplicial complex, where $d=\vert V(\cP)\vert/2=\vert \cP\vert$. From \cite[Theorem 5.1.13]{Bruns_Herzog} and \cite[Theorem 2.19 (b)-(c)]{binomial ideals}, we get the claim. 
\end{proof}

\section{Rook polynomial of closed paths and weakly closed paths} \label{Section: Rook polynomial - final}

In this section, we explore the relationship between the $h$-polynomial and the rook polynomial of a polyomino. Firstly, we introduce some basics regarding the Hilbert-Poincar\'{e} series of a graded $K$-algebra $R/I$. Consider a graded $K$-algebra $R$ and an homogeneous ideal $I$ of $R$. $R/I$ has a natural structure of graded $K$-algebra as $\bigoplus_{k\in\mathbb{N}}(R/I)_k$. The \textit{Hilbert-Poincar\'e series} of $R/I$ is the formal series $\rHP_{R/I}(t)=\sum_{k\in\mathbb{N}}\dim_{K} (R/I)_kt^k$. According to the celebrated Hilbert-Serre Theorem, there exists a unique polynomial $h(t)\in \mathbb{Z}[t]$, called \textit{h-polynomial} of $R/I$, such that $h(1)\neq0$ and $\rHP_{R/I}(t)=\frac{h(t)}{(1-t)^d}$, where $d$ is the Krull dimension of $R/I$. Recall that if $R/I$ is Cohen-Macaulay then $\mathrm{reg}(R/I)=\deg h(t)$. For simplicity, if $\cP$ is a polyomino, we denote the $h$-polynomial of $K[\cP]$ by $h_{K[\cP]}(t)$ and it is sometimes referred to as the \textit{h-polynomial of $\cP$}.\\
We will now describe a combinatorial interpretation of the $h$-polynomial for a closed path with a zig-zag walk, which follows from Theorem \ref{Thm: shelling} and the McMullen-Walkup Theorem (see \cite[Corollary 5.1.14]{Bruns_Herzog}).\\

\begin{corollary}\label{Coro: Walkup-McMullen}
	Let $\cP$ be a closed path with a zig-zag walk. Then the $i$-th coefficient of the $h$-polynomial of $K[\cP]$ is equal to the number of the facets of  $\Delta(\cP)$ having $i$ steps. \\
\end{corollary}

Now, let us introduce some definitions and concepts related to the rook polynomial of a polyomino $\cP$. Two rooks in $\cP$ are in \textit{non-attacking position} if they do not belong to the same row or column of cells of $\cP$. A \textit{$k$-rook configuration} in $\cP$ is a configuration of $k$ rooks arranged in $\cP$ in non-attacking positions. The maximum number of rooks that we can place in $\cP$ in non-attacking positions is called the \textit{rook number} and it is denoted by $r(\cP)$. Let $\cR(\cP,k)$ be the set of all $k$-rook configurations in $\cP$ and set $r_k=\vert \cR(\cP,k)\vert $ for all $k\in\{0,\dots,r(\cP)\}$ (conventionally $r_0=1$). The \textit{rook-polynomial} of $\cP$ is the polynomial in $\mathbb{Z}[t]$ defined as $r_{\cP}(t)=\sum_{k=0}^{r(\cP)}r_kt^k$. For example, if $\cP$ is a square tetromino then $r(\cP)=2$ and $r_{\cP}(t)=1+4t+2t^2$. If the readers wish to explore further, they can refer to \cite{L-convessi, Kummini rook polynomial, Parallelogram Hilbert series, romeo, Trento3}. \\

We are interested in interpreting the $h$-polynomial of a closed path with zig-zag walks in terms of its related rook polynomial. 
Let $\cP$ be a closed path containing a zig-zag walk. We define the following map $\phi$ between the set $\cF(\Delta_{\cP})_i$ of the facets of $\Delta(\cP)$ with $i$ steps and the set $\cR_i$ of the $i$-rook configurations of $\cP$ ($i\geq 0$). Let $F\in \cF(\Delta_{\cP})_i$. If $F=F_0$, then $\phi(F)=\emptyset$; otherwise, $\psi(F)=\{R_1,\dots,R_i\}$ where $R_j$ is a rook placed in a step cell of a step of $F$, for $1\leq j\leq i$. See, for instance, Figure \ref{Figure: exa of rook-facet}. 

	\begin{figure}[h]
	\centering
	\subfloat{\includegraphics[scale=0.7]{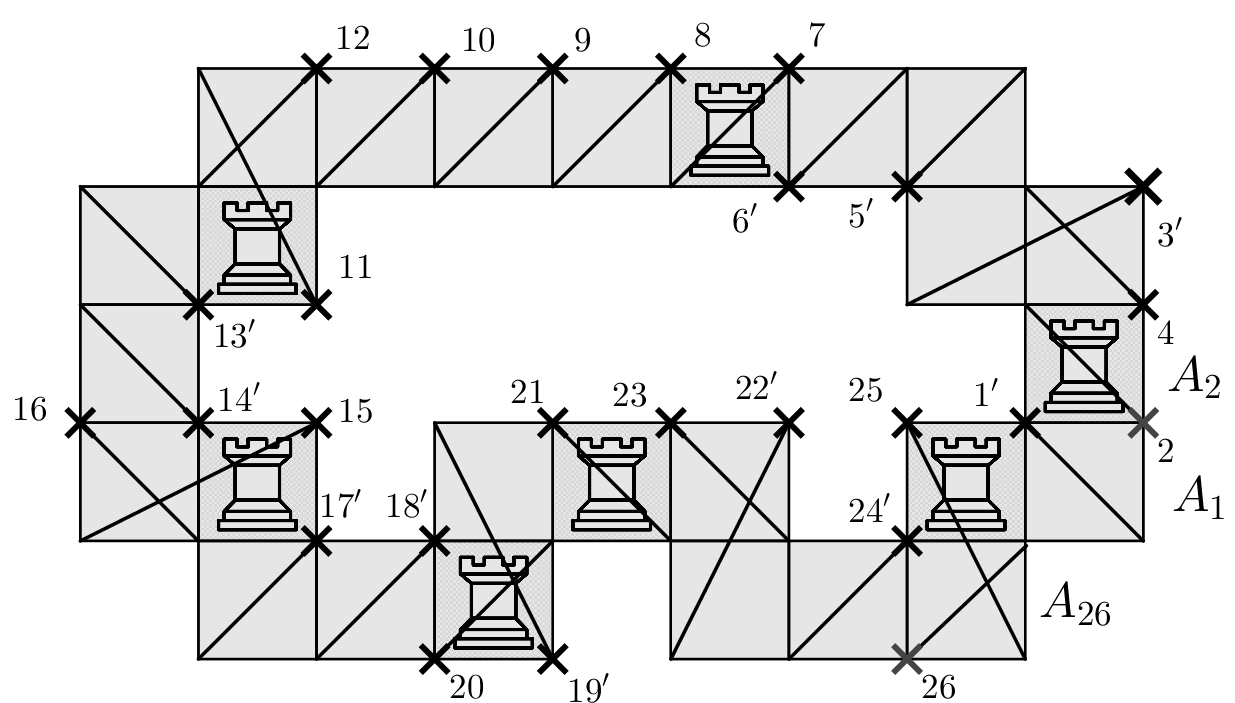}}
	\subfloat{\includegraphics[scale=0.7]{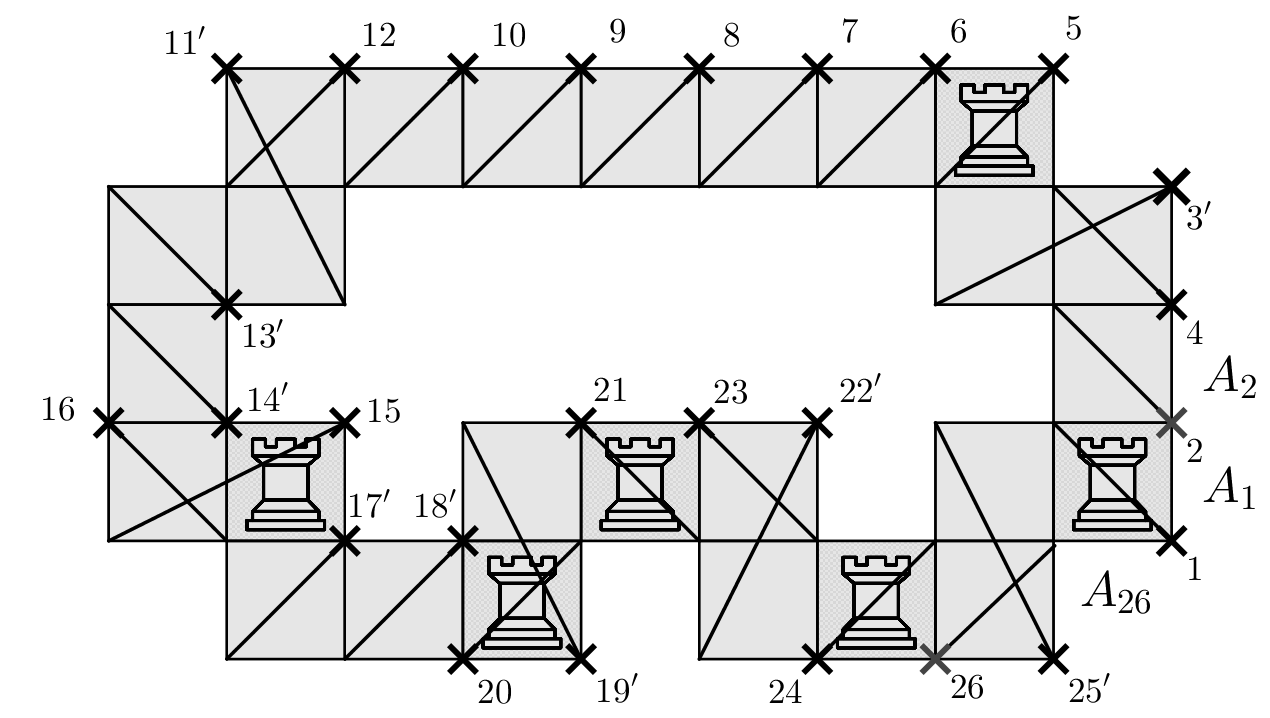}}
	\caption{Example of rook configurations and facets.}
	\label{Figure: exa of rook-facet}
\end{figure}

From Discussion \ref{Discussion}, it easily follow that $\phi$ is bijective. Consequently, the number of the facets of $\Delta(\cP)$ with $i$ steps equals the number of the $i$-rook configurations of $\cP$ (with $i\geq 0$). Therefore, by combining Corollaries \ref{Coro: Cohen-Macaulay} and \ref{Coro: Walkup-McMullen}, we obtain the following result.\\

\begin{proposition}\label{Prop: h-polynomial for zig-zag}
	Let $\cP$ be a closed path with a zig-zag walk. Then the $h_{K[\cP]}(t)$ is equal to the rook-polynomial of $\cP$. In particular, $\mathrm{reg}(K[\cP])$ is the rook number of $\cP$.  \\
\end{proposition}

 In conclusion, we achieve the following outcome that covers the entire class of closed paths. \\

\begin{theorem}\label{Thm: rook polynomial of ALL closed paths }
	Let $\cP$ be a closed path. Then the $h_{K[\cP]}(t)$ is equal to the rook-polynomial of $\cP$. In particular, $\mathrm{reg}(K[\cP])$ is the rook number of $\cP$. 
\end{theorem}

\begin{proof}
	If $\cP$ has no zig-zag walks then the claim follows from \cite[Theorem 5.5]{Cisto_Navarra_Hilbert_series}. Otherwise, we get the desired conclusion from Proposition \ref{Prop: h-polynomial for zig-zag}.\\
\end{proof}

Actually, the arguments used in the proofs of the results in this work can be extended to the class of \textit{weakly closed paths} (see \cite[Definition 4.1]{Cisto_Navarra_weakly}). If $\cP$ is a weakly closed path, we can perform rotations or reflections of $\cP$ in order that $\{A_n,A_1,A_2\}$ is as in Figure \ref{Figure: weakly}. To apply Algorithm \ref{Algorithm: to define Y}, we need to fix ${1, 2}$ and ${1', 2'}$ as starting points, depending on the position of $A_3$ relative to $A_2$. For instance, in Figure \ref{Figure: weakly} (c), we set $a=3$, $b=3'$, $c=2$ and $d=2'$ if $A_4$ is at North of $A_3$ or $a=2'$, $b=2$, $c=3'$ and $d=3$ if $A_4$ is at West of $A_3$. \\

 \begin{figure}[h!]
	\centering
	\subfloat[]{\includegraphics[scale=0.65]{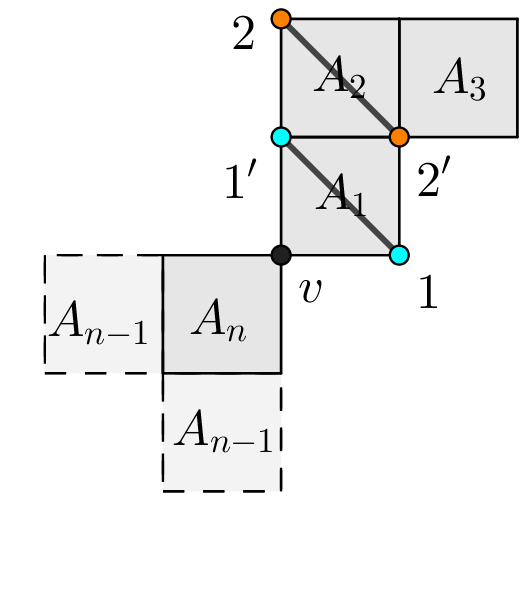}}\quad
	\subfloat[]{\includegraphics[scale=0.65]{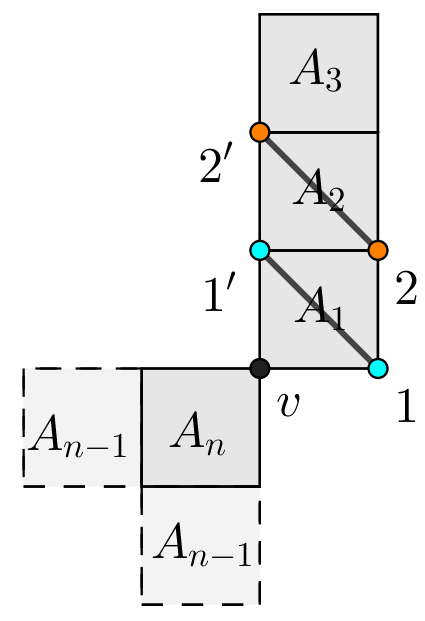}}\quad
	\subfloat[]{\includegraphics[scale=0.65]{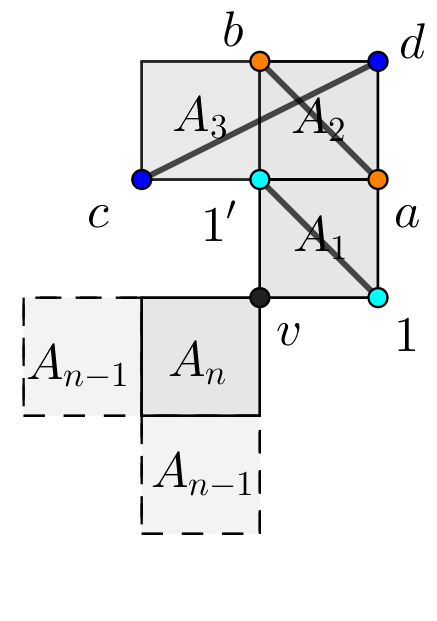}}
	\caption{Hooking parts of a weakly closed path.}
	\label{Figure: weakly}
\end{figure}

It is easy to see that $\vert V(\cP)\vert-1=2\vert \cP\vert$ and that the vertex $v$ may or may not labelled by applying Algorithm \ref{Algorithm: to define Y}, depending on the shape of $\cP$. Therefore, let $Y=Y_1\sqcup Y_2$ be the set of the labels given by Algorithm \ref{Algorithm: to define Y} and $V(\cP)\setminus Y=\{w\}$, where $w$ can be $v$ or a different vertex (note that we can identify the labels of the vertices with the vertices themselves). Take $<_{Y_2}$ an arbitrary total order on $Y_2$ and define the following total order on $\{x_v:v\in V(\cP)\}$. Let $i,j\in V(\cP)$.
\[
x_i <_\cP x_j  \Longleftrightarrow \left \lbrace \begin{array}{l}
	i \in Y_2 \text{ and } j\in Y_1\\
	i,j\in Y_1 \text{ and } i<j \\
	i,j \in Y_2 \text{ and } i<_{Y_2} j\\
	i=w \text{ and } j\neq w 
\end{array}
\right.
\]

\noindent Set by $\prec_{\cP}$ the lexicographic order on $S_{\cP}$ induced by the total order $<_{\cP}$.\\ 
Once the previous definitions are set, we can easily obtain the following result.\\

\begin{theorem}\label{Thm: weakly}
	Let $\cP$ be a weakly closed path polyomino and $\prec_{\cP}$ be the above lexicographic order. Hence:
	\begin{enumerate}
		\item The set of the generators of $I_{\cP}$ forms the reduced (quadratic) Gr\"obner basis of $I_{\cP}$ with respect to $\prec_{\cP}$.
		\item $I_{\cP}$ is a radical ideal and $K[\cP]$ is a Koszul ring.  
		\item If $\cP$ does not contains zig-zag walks, then $K[\cP]$ is a normal Cohen-Macaulay domain with Krull dimension equal to $\vert V(\cP)\vert-\vert\cP\vert$.
		\item $\Delta(\cP)$ is a shellable simplicial complex.
		\item If $\cP$ contains a zig-zag walk, then $K[\cP]$ is a Cohen-Macaulay ring and $\dim(K[\cP])=\vert V(\cP)\vert -\vert \cP\vert$.
		\item $h_{K[\cP]}(t)$ is equal to the rook-polynomial of $\cP$. In particular, $\mathrm{reg}(K[\cP])$ is the rook number of $\cP$. 
		\item $I_{\cP}$ is of K\"onig type.
	\end{enumerate} 
\end{theorem}

\begin{proof}
	1) Consider two inner intervals $I$ and $J$ of $\cP$ containing $A_1$ and $A_n$, respectively. Let $f$ and $g$ be the generators of $I_\cP$ attached to $I$ and $J$, respectively. Observe that $x_1$ divides $\mathrm{in}(f)$ and $x_1>x_i$ for all $i\in V(\cP)\setminus\{1\}$, so $\mathrm{gcd}(\mathrm{in}(f),\mathrm{in}(g))=1$. all the other cases can be proved as in Proposition \ref{Prop: Grobner basis}.\\
	2) It follows from 1).\\
	3) It can be proved as in Corollary \ref{Coro: radicality + CM closed path no zig-zag}, once we observe that $I_{\cP}$ is a toric ideal from \cite[Theorem 4.7]{Cisto_Navarra_weakly}.\\
	4-5) The arguments provided in Theorem \ref{Thm: shelling} and Corollary \ref{Coro: Cohen-Macaulay} can be used in a similar way for the weakly closed paths.\\
	6) If $\cP$ has no zig-zag walks then $\cP$ contains an $L$-configuration, or a weak $L$-configuration, or a ladder of at least three steps or a weak ladder (see \cite[Proposition 4.5]{Cisto_Navarra_weakly}). If $\cP$ has an $L$-configuration or a ladder of at least three steps, then the claim follows using similar arguments as done in \cite[Sections 3 and 4]{Cisto_Navarra_Hilbert_series} and \cite[Corollary 2.5]{Cisto_Navarra_Jahangir}. If $\cP$ has a weak $L$-configuration or a weak ladder, then we can apply the strategy used in \cite[Section 4]{Cisto_Navarra_Hilbert_series} (see also \cite[remark 5.8]{Cisto_Navarra_Hilbert_series}). When $\cP$ has a zig-zag walk, then we get the desired conclusion as done in Proposition \ref{Prop: h-polynomial for zig-zag}.\\
	7) The conditions in \cite[Definition 1.1]{Def. Konig type} are satisfied taking the monomial order $\prec_{\cP}$ and the generators of $I_{\cP}$ whose initial monomial with respect $\prec_{\cP}$ is given by $x_ix_{i'}$ where $i\in Y_1$ and $i'\in Y_2$ ($Y_1$ and $Y_2$ are defined by using Algorithm \ref{Algorithm: to define Y}), since $\mathrm{ht}(I_{\cP})=\vert \cP\vert$.\\
\end{proof}

\begin{remark}\rm \label{Remark: Final}
 Let $\cP$ be a polyomino. Does there exist a monomial order on $S_{\cP}$ such that the reduced (quadratic) Gr\"obner basis of $I_{\cP}$ consists of the set of the generators of $I_{\cP}$ and the simplicial complex $\Delta(\cP)$ attached to $\cP$ is shellable? This question is affirmed for some classes of polyominoes, such as frame polyominoes \cite{Frame}, grid polyominoes \cite{Dinu_Navarra_grid} and (weakly) closed paths. If this holds for all polyominoes, then the strategy used in this work (as well as in \cite{Dinu_Navarra_grid} and \cite{Frame}), which involves studying the shelling order of the facets of $\Delta(\cP)$ and applying the McMullen-Walkup Theorem (\cite[Corollary 5.1.14]{Bruns_Herzog}), might be useful for addressing \cite[Conjecture 3.2]{Parallelogram Hilbert series} (or its generalization \cite[Conjecture 4.9]{Frame}) or \cite[Conjecture 4.5]{Trento3} for thin polyominoes. This could also imply that $K[\cP]$ is Cohen-Macaulay for every polyomino $\cP$. By combining this result with \cite[Corollary 3.1.17]{Villareal} and \cite[Theorem 1.1]{Moradi}, we could provide a positive answer to \cite[Conjecture 3.7]{Dinu_Navarra_Konig}, which states that $\mathrm{ht}(I_{\cP})=\vert\cP\vert$ for every polyomino $\cP$.\\
 However, it seems that the shelling order and what is termed as \textit{steps} are highly dependent on the polyomino’s shape; for instance, different definitions than Definition \ref{Defn: steps} can be found in \cite[Definition 3.3]{Frame} and \cite[Definition 3.1]{Dinu_Navarra_grid}, depending on the specific polyominoes studied. Therefore, identifying a more general framework or something more general than the so-called steps for describing a shelling order for a simplicial complex attached to a polyomino remains an open question. \\
\end{remark}

\bmhead{Acknowledgements}
The author acknowledges the support of the Scientific and Technological Research Council of Turkey (T\"UB\.ITAK) under Grant No. 122F128 and expresses his gratitude to T\"UB\.ITAK for their generous support. Additionally, he states that he is a member of the GNSAGA group of INDAM and he is grateful for its support.\\
The inspiration for this work arose from the insightful suggestions provided by an anonymous referee regarding \cite[Section 2]{Dinu_Navarra_Konig} during the review of that draft. The author wishes to express his gratitude for the valuable advice offered.\\
He also wishes to thank Ayesha Asloob Qureshi for her insightful and significant discussions and support.\\
This article is dedicated to celebrating the \textit{II Meeting UMI for Doctoral Students}, which took place in Naples on June 13-14, 2024. The author extends his sincere thanksgivings to the Scientific Committee for the opportunity to present his talk titled \textit{A Combinatorial Interpretation of the Primary Decomposition of Binomial Ideals Attached to Polyocollections}.\\

\begin{footnotesize}
\noindent \textbf{Data Availability.} There is no data to be made available.\\
\noindent \textbf{Conflict of interest.} The author states that there is no conflict of interest.	
\end{footnotesize}

\end{document}